\documentclass[reqno,12pt]{amsart}

\usepackage{graphicx}
\usepackage{hyperref}
\usepackage{psfrag}
\usepackage{epsfig,color}

\newcommand{\be}{\begin{eqnarray}}
\newcommand{\ee}{\end{eqnarray}}
\newcommand{\ben}{\begin{eqnarray*}}
\newcommand{\een}{\end{eqnarray*}}
\newcommand{\dis}{\displaystyle}
\newcommand{\beq}{\begin{equation}}
\newcommand{\noi}{\noindent}
\newcommand{\eeq}{\end{equation}}
\newcommand{\R}{{\mathbb R}}
\newcommand{\N}{{\mathbb N}}

\newcommand{\vs}{{\vskip .2cm}}

\newcommand{\sep}{; }

\newtheorem{thm}{Theorem}[section]
\newtheorem{coro}[thm]{Corollary}
\newtheorem{prop}[thm]{Proposition}
\newtheorem{lemma}[thm]{Lemma}

\renewcommand{\subjclass}[1]{\thanks{2010 \emph{Mathematics Subject Classification.} #1}}

\title[Nodal solutions and $p$-Laplacian]{Qualitative properties and existence of sign changing solutions with compact support for an equation with a $p$-Laplace operator}

\author[{Jean Dolbeault, Marta Garc\'ia-Huidobro and Raul Man\'asevich}]{}


\thanks{
${}^1$ Partially supported by ECOS-Conicyt projects no. C11E07, and by the CMM (UMR CNRS no. 2071), Universidad de Chile. \\ \indent
${}^2$ Partially supported by FONDECYT GRANT 1110268.\\ \indent
${}^3$ Partially supported by Basal-CMM-Conicyt, Milenio grant-P05-004F, and FONDECYT GRANT {1110268}.}
\email{J.~Dolbeault: dolbeaul@ceremade.dauphine.fr,\newline M.~Garc\'ia-Huidobro: mgarcia@mat.puc.cl, R.~Man\'asevich: manasevi@dim.uchile.cl}

\begin{document}
\maketitle
\thispagestyle{empty}
\centerline{\scshape Jean Dolbeault$\,{}^1$}\vskip 0.02cm
{\footnotesize
\centerline{Ceremade (UMR CNRS no. 7534), Universit\'e Paris-Dauphine,}
\centerline{Place de Lattre de Tassigny, 75775 Paris C\'edex~16, France}
}\vskip 0.3cm
\centerline{\scshape Marta Garc\'ia-Huidobro$\,{}^2$}\vskip 0.02cm
{\footnotesize
\centerline{Facultad de Matem\'aticas}
\centerline{Pontificia Universidad Cat\'olica de Chile, Casilla 306 Correo 22, Santiago, Chile}
}\vskip 0.3cm
\centerline{\scshape R\'aul Man\'asevich$\,{}^3$}\vskip 0.02cm
{\footnotesize
\centerline{DIM \& CMM (UMR CNRS no. 2071), FCFM,}
\centerline{Universidad de Chile, Casilla 170 Correo 3, Santiago, Chile}
\vskip 0.5cm
\centerline{\bf This paper is dedicated to Klaus Schmitt}
\centerline{\today}
}\vspace*{0.5cm}

\begin{quote}{\normalfont\fontsize{10}{12}\selectfont{\bfseries Abstract.} We consider radial solutions of an elliptic equation involving the $p$-Laplace operator and prove by a shooting method the existence of compactly supported solutions with any prescribed number of nodes. The method is based on a change of variables in the phase plane corresponding to an asymptotic Hamiltonian system and provides qualitative properties of the solutions.\par}\end{quote}

%%%%%%%%%%%%%%%%%%%%%%%%%%%%%%%%%%%%%%%%%%%%%%%%%%%%%%%%%%%%%%%%%%%%%%
%%%%%%%%%%%%%%%%%%%%%%%%%%%%%%%%%%%%%%%%%%%%%%%%%%%%%%%%%%%%%%%%%%%%%%
\section{introduction}\label{Section:Intro}

In this paper we shall consider classical radial sign-changing solutions of
\beq\label{eq1}
\Delta_pu+f(u)=0
\eeq
on $\R^N$ with $p>1.$ Radial solutions to~(\ref{eq1}) satisfy the problem
\beq\label{eqq2}
\left(r^{N-1}\,\phi_p(u')\right)'+r^{N-1}\,f(u)=0\;,\quad u'(0)=0\;.
\eeq
Here, for any $s\in\R\setminus\{0\}$, $\phi_p(s):=|s|^{p-2}\,s$ and $\phi_p(0)=0.$ Also $\phantom{}'$ denotes the derivative with respect to $r=|x|\geq 0$, $x\in\R^N$ and for radial functions as it is usual we shall write $u(x)=u(r)$. We will assume henceforth that $N>p.$ By a (classical) solution of~(\ref{eqq2}), we mean a function $u$ in~$C^1([0,\infty))$ such that $u'(0)=0$ and $|u'|^{p-2}\,u'$ is in~$C^1(0,\infty)$.
\vs
It is well known that equations involving quasilinear operators ($p$-Laplace, mean curvature) may have positive solutions with compact support, see for example \cite{Conti-Gazzola}, \cite{Gazzola-Serrin-Tang}, and \cite{GarciaHuidobro-Manasevich-Serrin-Tang-Yarur01}. We are interested here in qualitative properties of the solutions to problem~(\ref{eqq2}) that have a prescribed number of zeros. They satisfy the problem
\begin{equation}\label{eq2}
\begin{gathered}
\left(r^{N-1}\,\phi_p(u')\right)'+r^{N-1}\,f(u)=0\;,\quad r>0\;,\\
u'(0)=0\;,\quad \lim\limits_{r\to\infty}u(r)=0\;.
\end{gathered}
\end{equation}
As we shall see in Section~\ref{Section:CSP}, under condition (H3) below, such solutions have compact support.
\vs
We assume the following conditions on $f$.
\vs
\begin{itemize}
\item [(H1)] $f$ is continuous on $\mathbb R$, locally Lipschitz on $\mathbb R\setminus\{0\}$, with $f(0)=0$.
\vs
\item [(H2)] There exist two constants $a>0$ and $b<0$ such that $f$ is strictly decreasing on $(b,a)$, and $a$ (resp. $b$) is a local minimum (resp. local maximum) of $f$.
\vs
\item [(H3)] The function $F(u):=\int_0^uf(s)\,ds$ is such that $u\mapsto |F(u)|^{-1/p}$ is locally integrable near $0$. More generally, we will assume that the function $u\mapsto |F(x_0)-F(u)|^{-1/p}$ is locally integrable near $x_0\neq 0$ whenever $x_0$ is a local maximum of $F$.
\vs
\item [(H4)] For any $u_0$ such that $f(u_0)=0$, $F(u_0)<0$.
\vs
\item [(H5)] The function $u\mapsto f(u)$ is nondecreasing for large values of $u$ and
satisfies
\[
\liminf\limits_{|u|\to\infty}\frac{f(u)}{|u|^{p-2}\,u}=\infty\;.
\]
\item [(H6)] For some $\theta\in(0,1)$, we have
\[
\liminf_{|x|\to\infty}\frac{F(\theta\,x)}{x\,f(x)}>\frac{N-p}{Np}>0\;.
\]
\end{itemize}
\vs
By the last two conditions, our problem is ($p$)-superlinear and subcritical. As a consequence of the previous assumptions, there exist two constants $B<0<A$ such that
\begin{itemize}
\item[$(i)$] $F(s)<0$ for all $s\in(B,A)\setminus\{0\}$, $F(B)=F(A)=0$ and $f(s)>0$ for all $s>A$ and $f(s)<0$ for all $s<B$,\vs
\item[$(ii)$] $F$ is strictly increasing in $(A,\infty)$ and strictly decreasing in $(-\infty,B)$,\vs
\item[$(iii)$] $F(s)$ is bounded below by $-\,\bar F=\min_{s\in[B,A]}F(s)$ for some $\bar F>0$,\vs
\item[$(iv)$] $\lim\limits_{|s|\to\infty}F(s)=\infty$.
\end{itemize}\vs
This paper is organized as follows. Our approach is based on a shooting method and a change of variables which is convenient to count the number of nodes. In Section~\ref{Section:CSP} we state and prove a version of the compact support principle for sign changing solutions. In Section~\ref{Section:Properties}, we consider the initial value problem~(\ref{ivp}), and establish some qualitative properties of the solutions. Most of these properties are interesting by themselves: see for instance Theorem~\ref{limexis}. Section~\ref{Section:Coordinates} is devoted to \emph{generalized polar coordinates} that allows us to write the initial value problem~(\ref{ivp}) as a suitable system of equations, see~(\ref{Syst:rho-theta}), that describes the evolution on the \emph{phase space} for the asymptotic Hamiltonian system corresponding to the limiting regime as $r\to\infty$. From this system we can estimate the number of rotations of solutions around the origin, in the phase space at high levels of the energy and relate it with the number of sign changes of the solution of~\eqref{eq2}. In Section~\ref{Section:Existence} we state and prove our existence results that essentially says that for any $k\in \N$ there is a solution to~(\ref{eq2}) with $k$ nodes that has compact support. This result differs from \cite{Balabane-Dolbeault-Ounaies01} in the sense that it holds for the $p$-Laplace operator for any $p>1$ and the nonlinearity $f$ is an arbitrary superlinear and subcritical function satisfying assumptions $(\mathrm H1)-(\mathrm H6)$. It also differs from the recent results of \cite{Cortazar-GarciaHuidobro-Yarur} in the sense that the change of coordinates of Section~\ref{Section:Coordinates} gives a detailed qualitative description of the dependence of the solutions in the shooting parameter $\lambda=u(0)$. When $\lambda$ varies, the number of nodes changes of at most one and we can estimate the size of the support of compactly supported solutions: see Section~\ref{Section:Qualitative} for more details and precise statements. Finally we state two already known results in the Appendix, for completeness. The first one deals with existence of solutions to the initial value problem~(\ref{ivp}) on $[0,\infty)$. The second one shows where uniqueness of the flow defined by \eqref{Syst:rho-theta} holds on the phase space; for a proof we refer to~\cite{Cortazar-GarciaHuidobro-Yarur}.

\medskip

The case $p=2$ has been studied in \cite{Balabane-Dolbeault-Ounaies01} for a special nonlinearity. Assumption~$(\mathrm H3)$ is the sharp condition for the existence of solutions with compact support; see \cite{Pucci-Serrin-Zou99}. If $u\mapsto|F(u)|^{-1/p}$ is not locally integrable, then Hopf's lemma holds according to \cite{Vazquez84}, and there is no solution with compact support. How to adapt the known results on the \emph{compact support principle} to solutions that change sign is relatively easy by extending the results of \cite{Serrin-Zou99}. See \cite{Benilan-Brezis-Crandall75,Cortazar-Elgueta-Felmer96,Balabane-Dolbeault-Ounaies01} in case $p=2$ and \cite{Vazquez84,Serrin-Zou99,Pucci-Serrin-Zou99,Pucci-Serrin00,Felmer-Quaas01,GarciaHuidobro-Manasevich-Serrin-Tang-Yarur01} in the general case.

We shall refer to~\cite{GMZ} and to~\cite{Cortazar-GarciaHuidobro-Yarur} respectively for multiplicity and existence results; earlier references can be found in these two papers. Consequences of a possible asymmetry of $F$ are not detailed here: see, \emph{e.g.,} \cite{Fabry-Manasevich02} for such questions. There is a huge literature on sign changing solutions and we can quote~\cite{MR2825334,MR2527543,MR2833356,MR2496356,MR2663307,MR2381665,MR2122573,MR2288443,MR2036055,MR2452116,MR2282828} for results in this direction, which are based either on shooting methods or on bifurcation theory but do not take advantage of the representation of the equation in the generalized polar coordinates.

Our main tool in this paper is indeed the change of variables of Section~\ref{Section:Coordinates}, which can be seen as the canonical change of coordinates corresponding either to \hbox{$N=1$} and $f(u)=|u|^{p-2}\,u$, or to the asymptotic Hamiltonian system in the limit $r\to +\infty$: see \cite{Fabry-Manasevich02,Drabek-Girg-Manasevich01} for earlier contributions.

\medskip We end this introduction with a piece of notation and some definitions. We shall denote by $F$ and $\Phi_p$ the primitives of $f$ and $\phi_p$ respectively, such that \hbox{$F(0)=\Phi_p(0)=0$}. Thus for any~\hbox{$u\in\R$},
\[
F(u)=\int_0^uf(s)\;ds\quad\text{and}\quad\Phi_p(u)=\frac 1p\,|u|^p\;.
\]
We shall say that the function $u$ has a \emph{double zero} at a point $r_0$ if $u(r_0)=0$ and $u'(r_0)=0$ simultaneously. We call \emph{nodes} of a solution
the zeros which are contained in the interior of the support of the solution and where the solution changes sign: for instance, a solution with zero node is a nonnegative solution, eventually with compact support.

%%%%%%%%%%%%%%%%%%%%%%%%%%%%%%%%%%%%%%%%%%%%%%%%%%%%%%%%%%%%%%%%%%%%%%
%%%%%%%%%%%%%%%%%%%%%%%%%%%%%%%%%%%%%%%%%%%%%%%%%%%%%%%%%%%%%%%%%%%%%%
\section{Compact support principle}\label{Section:CSP}

The following result is an extension to sign changing solutions of the \emph{compact support principle,} which is usually stated only for nonnegative solutions. See for instance \cite{Cortazar-Elgueta-Felmer96, Pucci-Serrin-Zou99}. Our result shows a compact support property of all solutions converging to~$0$ at infinity, without sign condition and generalizes a result for the case $p=2$ that can be found in \cite{Balabane-Dolbeault-Ounaies01}.
\vs
%---------------------------------------------------------------------
\begin{lemma}\label{compact_support} Assume that $f$ satisfies assumptions $(\mathrm H1)$, $(\mathrm H2)$ and $(\mathrm H3)$. Then any bounded solution $u$ of~(\ref{eq2}) has compact support. \end{lemma}
%---------------------------------------------------------------------
\vs
\begin{proof} Let us set
\[
\mathcal A=\int_0^a\frac{ds}{\left[p'\,(-F(s))\right]^{\frac 1p}}\;,
\]
where $p'=p/(p-1)$. Defining $\bar u$ on $(0,\mathcal A)$ implicitly by
\[
r=\int_{\bar u(r)}^a\frac{ds}{\left[p'\,(-F(s))\right]^{\frac 1p}}\;,
\]
we first have that
\[
\frac 1{p'}|\bar u'|^p+F(\bar u)=0\;,
\]
and, by differentiation, that $\bar u$ satisfies
\[
(\phi_p({\bar u}'))'+f(\bar u)=0\;.
\]
It is straightforward to check that $\bar u(0)=a$ and $\bar u(\mathcal A)=0$, so that $\bar u'(\mathcal A)=0$ as well. We may then extend $\bar u$ to $(\mathcal A,+\infty)$ by $0$.
\vs
Let $u$ be a bounded solution of~(\ref{eq2}) such that $\lim\limits_{r\to\infty}u(r)=0$. Then there exists $R>0$ such that
\[
b<u(r)<a\quad\forall\;r>R\;.\]
Let
\[w(r):=\bar u(r-R)\quad\forall\;r\geq R\;.\]
Then either $u(r)\leq w(r)$ for any $r\geq R$, and, as a consequence, $u(r)\leq w(r)\leq 0$ for any $r\geq R+\mathcal A$, or there exists $r_0>R$ such that $u(r_0)>w(r_0)$. Assume that this last case holds. Since $(u-w)(R)\leq 0$ and $\lim_{r\to\infty}(u-w)(r)=0$, with no restriction we may assume that
\[
(u-w)(r_0)=\max_{r\in[R,\infty)}(u-w)>0\;.
\]
Hence, there exists a positive $\varepsilon$ such that
\[(u-w)(r)>0\quad\forall\;r\in[r_0,r_0+\varepsilon)\;.
\]
{}From the equations satisfied by $u$ and~$w$,
\begin{eqnarray*}
&&(r^{N-1}\,\phi_p(u'))'+r^{N-1}\,f(u)=0\;,\\
&&(r^{N-1}\,\phi_p(w'))'+r^{N-1}\,f(w)=(N-1)\,r^{N-2}\,\phi_p(w')\;,
\end{eqnarray*}
by integrating from $r_0$ to $r\in(r_0,r_0+\varepsilon)$, and by taking into account the fact that \hbox{$(u-w)'(r_0)=0$}, we get
\begin{eqnarray*}
&&r^{N-1}\,\phi_p(u'(r))-r^{N-1}\,\phi_p(w'(r))\\
&&=-\int_{r_0}^rs^{N-1}\underbrace{\Big(f(u(s))-f(w(s))\Big)}_{< 0\;\mbox{because}\;u(s)>w(s)}\;ds-(N-1)\int_{r_0}^rs^{N-2}\kern -13pt\underbrace{\phi_p(w'(s))}_{\leq 0\;\mbox{because}\;w'\leq 0}\kern -13pt\;ds\;,
\end{eqnarray*}
which proves that $u'>w'$ on $(r_0,r_0+\varepsilon)$. This obviously contradicts the assumption that $u-w$ achieves its maximum at $r=r_0$.

Summarizing, we have proved that $u(r)\leq w(r)$ for any $r\geq R$, and, as a consequence,
\[u(r)\leq 0\quad\forall\;r\geq R+\mathcal A\;.\]

\medskip Similarly, we observe that $\tilde u(r):=-u(r)$ is a solution of
\[
(r^{N-1}\,\phi_p(\tilde u'))'+r^{N-1}\tilde f(\tilde u)=0\;,\quad \tilde u'(0)=0\;,\quad\lim_{r\to\infty}\tilde u(r)=0\;,
\]
where
\[
\tilde f(s):=-f(-s)
\]
has the same properties as $f$, except that the interval $(b,a)$ has to be replaced by the interval $(-a,-b)$. With obvious notations, we obtain that
\[
\tilde u(r)\leq\tilde w(r)\quad\forall\;r\geq R+\mathcal B\;,
\]
for a certain positive $\mathcal B$ and where $\tilde w$ is a nonnegative solution of
\[
(\phi_p(w'))'+\tilde f(w)=0\quad\text{on}\quad(R,R+\mathcal B)\;,
\]
such that $\tilde w(R)=-b$, $\tilde w(R+\mathcal B)=\tilde w'(R+\mathcal B)=0$, and $\tilde w(r)=0$ for any $r\geq R+\mathcal B$. This proves that
\[
u(r)\geq 0\quad\forall\;r\geq R+\mathcal B\;,
\]
which completes the proof:
\[
u\equiv 0\quad\text{on}\quad (R+\max\{\mathcal A,\mathcal B\},\infty)\;.
\]
\end{proof}

%%%%%%%%%%%%%%%%%%%%%%%%%%%%%%%%%%%%%%%%%%%%%%%%%%%%%%%%%%%%%%%%%%%%%%
%%%%%%%%%%%%%%%%%%%%%%%%%%%%%%%%%%%%%%%%%%%%%%%%%%%%%%%%%%%%%%%%%%%%%%
\section{Properties of the solutions}\label{Section:Properties}

To deal with problem~\eqref{eq2}, we will use a shooting method and consider the initial value problem
\vs
\begin{equation}\label{ivp}
\begin{gathered}
\left(r^{N-1}\,\phi_p(u')\right)'+r^{N-1}\,f(u)=0\;,\quad r>0\;,\\
u(0)=\lambda>0\;,\quad u'(0)=0\;.
\end{gathered}
\end{equation}
To emphasize the dependence of the solution to \eqref{ivp} in the shooting parameter $\lambda$, we will denote it $u_\lambda$. Solutions to \eqref{ivp} exist and are globally defined on $[0,\infty)$; see a proof of this fact in Appendix~\ref{Section:Appendix}. By Proposition~\ref{unique-ext}, these solutions are uniquely defined until they reach a double zero or a point $r_0$ with $u'(r_0)=0$ and such that $u(r_0)$ is a relative maxima of~$F$.

To be used in our next results, to a solution $u_\lambda(r)$ of \eqref{eqq2}, we associate the energy function
\beq\label{funct-10}
E_\lambda(r):=\frac{|u_\lambda'(r)|^p}{p'}+F(u_\lambda(r))\;,
\eeq
where $p'=p/(p-1)$. The following proposition shows several properties of the solution $u_\lambda$ to \eqref{ivp} that are needed to prove Theorem~\ref{Main}.

\vs
%---------------------------------------------------------------------
\begin{prop}\label{basic1} Let $f$ satisfy $(\mathrm H1)$ through $(\mathrm H5)$ and let $u_\lambda$ be a solution of \eqref{ivp}.
\begin{enumerate}
\item[$(i)$] The energy $E_\lambda$ is nonincreasing and bounded, hence the limit
\[
\lim_{r\to\infty}E_\lambda(r)=\mathcal E_\lambda
\]
is finite.
\item[$(ii)$] There exists $C_{\lambda}>0$ such that $|u_\lambda(r)|+|u_\lambda'(r)|\le C_{\lambda}$ for all $r\ge 0$.
\item[$(iii)$] If $u_\lambda$ reaches a double zero at some point $r_0>0$, then $u_\lambda$ does not change sign on $[r_0,\infty)$. Moreover, if $u_\lambda\not\equiv 0$ for $r\ge r_0$, then there exists $r_1\ge r_0$ such that $u_\lambda(r)\neq0$, and $E_\lambda(r)<0$ for all $r>r_1$ and $u_\lambda\equiv 0$ on $[r_0,r_1]$.
\item[$(iv)$]If $\lim_{r\to\infty}u_\lambda(r)$ exists, then there exists a zero $\ell$ of $f$ such that
\[
\lim_{r\to \infty}u_\lambda(r)=\ell\quad\mbox{and}\quad \lim_{r\to \infty}u_\lambda'(r)=0\;.
\]
\end{enumerate}
\end{prop}
%---------------------------------------------------------------------
\begin{proof} Let $u_\lambda(r)$ be any solution of \eqref{ivp}. As
\begin{eqnarray*}\label{I'}
E_\lambda'(r)=-\,\frac{(N-1)}{r}\,|u_\lambda'(r)|^p\;,
\end{eqnarray*}
and $N\ge p>1$, we have that $E_\lambda$ is decreasing in $r$. Moreover, we have that
\[
F(\lambda)\ge F(u_\lambda(r))\ge -\,\bar F
\]
and thus $(i)$ and $(ii)$ follow by recalling that from $(\mathrm H5)$ we get $\lim_{|s|\to\infty}F(s)=\infty$.

Assume next that $u_\lambda$ reaches a double zero at some point $r_0>0$. Then $E_\lambda(r_0)=0$ implying that $E_\lambda(r)\le 0$ for all $r\ge r_0$. If $u_\lambda$ is not constantly equal to $0$ for  $r\ge r_0$, then $E_\lambda(r_1)<0$ for some $r_1>r_0$ and thus, by the monotonicity of $E_\lambda$, $E_\lambda(r)<0$ for all $r\ge r_1$. Moreover $u_\lambda$ cannot have the value $0$ again (because at the zeros of $u_\lambda$ we have $E_\lambda\ge 0$). This proves $(iii)$ by taking the infimum on all $r_1$ with the above properties.

Finally, if $\lim_{r\to\infty}u_\lambda(r)=\ell$, then from the equation in \eqref{ivp} and applying L'H\^opital's rule twice, we obtain that
\begin{eqnarray*}
0=\lim_{r\to\infty}\frac{u_\lambda(r)-\ell}{r^{p'}}&=& -\lim_{r\to\infty}\frac{r^{\frac{N-1}{p-1}}\,|u_\lambda'(r)|}{p'r^{\frac{N-1}{p-1}}r^{p'-1}}\\
&=& -\,\frac1{p'}\,\Bigl(\lim_{r\to\infty}\frac{r^{N-1}\,|u_\lambda'(r)|^{p-1}}{r^N}\Bigr)^{p'-1}\\
&=& -\,\frac1{p'}\,\Bigl(\lim_{r\to\infty}\frac{r^{N-1}\,f(u_\lambda)}{N\,r^{N-1}}\Bigr)^{p'-1}=-\,\frac1{p'}\,\Bigl(\frac{f(\ell)}{N}\Bigr)^{p'-1}\;.
\end{eqnarray*}

Next, from the definition in (\ref{funct-10}), it follows that $\lim_{r\to \infty} |u'(r)|=\big(p'\,(\mathcal E_\lambda-F(\ell)\big)^{1/p}$. Assume that $\lim_{r\to \infty} |u'(r)|:=m>0.$ Then given $0<\varepsilon <m$ there is $r_0>0$ such that $u'(r)>m-\varepsilon>0$ or $u'(r)<-m+\varepsilon<0$, for all $r\ge r_0$. Hence either $u(r)>u(r_0)+(m-\varepsilon)(r-r_0)$ or $u(r)<u(r_0)+(-m+\varepsilon)(r-r_0)$, for all $r> r_0$, which is impossible because $\lim_{r\to\infty}E_\lambda(r)=\mathcal E_\lambda$ is finite, and $(iv)$ follows.
\end{proof}\vs

%---------------------------------------------------------------------
\begin{prop}\label{basic2} Let $f$ satisfy $(\mathrm H1)$-$(\mathrm H5)$ and let $u_\lambda$ be a solution of \eqref{ivp}.
Then $u_\lambda$ has at most a finite number of sign changes. \end{prop}
%---------------------------------------------------------------------
\begin{proof} The result is true if $u$ reaches a double zero. Let us prove it by contradiction. If $\{z_n\}$ is a sequence of zeros accumulating at some double zero $r_0$, then  for each $n\in\mathbb N$, there exists a unique point $r_n\in(z_n,z_{n+1})$ at which $u_\lambda$ reaches its maximum or minimum value. At these points, using that $E_\lambda(r_n)\ge E_\lambda(z_n)\ge0$,  we must have that
\[
|u_\lambda(r_n)|\ge \min\{|B|, A\}.
\]
As we also have that $u_\lambda(r_n)\to u_\lambda(r_0)=0$, we obtain a contradiction.

This proves that $u_\lambda$ has only a finite number of zeros on $(0,r_0)$, and by Proposition~\ref{basic1}$(iii)$, we know that $u_\lambda$ cannot change sign on $(r_0,\infty)$. Hence, without loss of generality we may assume that $u$ does not have any double zero. By the above argument, we also know that zeros cannot accumulate.

Next, we argue by contradiction and suppose that there is an infinite sequence (tending to infinity) of simple zeros of $u$. Then $E_\lambda(r)\ge 0$ for all $r>0$. We denote by $\{z_n^+\}$ the zeros for which $u'(z_n^+)>0$ and by $\{z_n^-\}$ the zeros for which $u'(z_n^-)<0$. We have
\[
0<z_1^-<z_1^+<z_2^-<\cdots<z_n^+<z_{n+1}^-<z_{n+1}^+<\cdots
\]
Between $z_n^-$ and $z_n^+$ there is a minimum $r_n^m$ where $u(r_n^m)<0$ and between $z_n^+$ and $z_{n+1}^-$ there is a maximum $r_n^M$ where $u(r_n^M)>0$. As $E_\lambda(r_n^M),\ E_\lambda(r_n^m)\ge 0$, it must be that $u(r_n^m)<B$ and $u(r_n^M)>A$.

\medskip We claim that \emph{there exists $T>0$ and $n_0\in\mathbb N$ such that the distance between two consecutive zeros is less than $T$ for all $n\ge n_0$}.

Indeed, let $a^+$ be the largest positive zero of $f$ ($b^-$ the smallest negative zero of $f$). Set
\[
d=A-a^+\;,\quad b_1=a^++\frac{d}4\;,\quad b_2=A-\frac{d}4\;.
\]

Let $r_{1,n}\in(z_n^+,r_n^M)$ be the unique point where $u(r_{1,n})=b_1$, and let $r_{2,n}\in(z_n^+,r_n^M)$ be the unique point where $u(r_{2,n})=b_2$. Then $z_n^+<r_{1,n}<r_{2,n}$. For $r\in(z_n^+,r_{2,n})$, $u(r)\in(0,b_2)\subset(0,B^+)$, hence $F(u(r))<0$ and thus
\[
\frac{|u'|^p}{p'}\ge |F(u(r))|
\]
implying that
\[
\frac{u'(r)}{|F(u(r))|^{1/p}}\ge (p')^{1/p}\quad \mbox{for all }r\in(z_n^+,r_{2,n})\;,
\]
and thus (from $(\mathrm H3)$)
\be\label{1}
\int_0^{b_2}\frac{du}{|F(u)|^{1/p}}\ge (p')^{1/p}\,(r_{2,n}-z_n^+)
\ee
Next, from the equation we have that for $r\in[r_{2,n},r_n^M]$,
\ben
|(\phi_p(u'))'(r)|&=&\Bigm|\frac{(N-1)}{r}\,\phi_p(u'(r))+f(u(r))\Bigm|\\
&\ge&f(u(r))-\frac{(N-1)}{r}\,\phi_p(C_{\lambda})\\
&\ge&f(b_2)-\frac{(N-1)}{r}\,\phi_p(C_{\lambda})\\
&\ge& \frac1{2}\,f(b_2)\quad\mbox{for all $r\ge\frac{2\,(N-1)\,\phi_p(C_{\lambda})}{f(b_2)}$}\;.
\een
Hence, choosing $n_0$ such that $z_n^+\ge \frac{2\,(N-1)\,\phi_p(C_{\lambda})}{f(b_2)}$ for all $n\ge n_0$, we have that
\[
|(\phi_p(u'))'(r)|\ge \frac1{2}\,f(b_2)\quad\mbox{for all }r\in[r_{2,n},r_n^M]
\]
and therefore
\[
\phi_p(C_{\lambda})\ge \phi_p(u'(r_{2,n}))-\phi_p(u'(r_n^M))=(\phi_p(u'))'(\xi)(r_n^M-r_{2,n})\ge\frac1{2}\,f(b_2)(r_n^M-r_{2,n})
\]
implying that
\be\label{2}
(r_n^M-r_{2,n})\le \frac{2\,\phi_p(C_{\lambda})}{f(b_2)}\;.
\ee
{}From \eqref{1} and \eqref{2} we conclude that
\[
r_n^M-z_n^+\le \frac1{(p')^{1/p}}\int_0^{b_2}\frac{du}{|F(u)|^{1/p}}+\frac{2\,\phi_p(C_{\lambda})}{f(b_2)}:=T_1\;.
\]
A similar argument over the interval $[r_n^M,z_{n+1}^-]$ yields
\[
z_{n+1}^--r_n^M\le T_1\;,
\]
implying
\[
z_{n+1}^--z_n^+\le 2\,T_1
\]
and finally, the same complete argument over the interval $[z_n^-,z_n^+]$ yields
\[
z_{n}^+-z_n^-\le 2\,\overline{T}_1
\]
for some $\overline{T}_1$ which will depend on $b^-$ and $B$ only and the claim follows with $T=\max\{2\,\overline{T}_1,2\,T_1\}$.

We can now prove the proposition. Observe that $u(r)\in[b_1,b_2]$ for $r\in[r_{1,n},r_{2,n}]$ and thus
\be
\label{3}|u'(r)|^p\ge p'\,|F(u(r))|\ge p'\,|F(b_2)|
\ee
and from the mean value theorem
\[
b_2-b_1\le C_{\lambda}\,(r_{2,n}-r_{1,n})\;,
\]
hence
\be\label{4}
r_{2,n}-r_{1,n}\ge\frac{b_2-b_1}{C_{\lambda}}\;.
\ee
Then,
\be\label{finite}
\infty>E_\lambda(z_{n_0}^+)-E_\lambda(\infty)&=&(N-1)\int_{z_{n_0}^+}^\infty\frac{|u'(t)|^p}{t}\;dt\\
&\ge&(N-1)\sum_{k=n_0}^\infty\int_{r_{1,k}}^{r_{2,k}}\frac{|u'(t)|^p}{t}\;dt\nonumber\\
\mbox{ from \eqref{3} }&\ge&(N-1)\sum_{k=n_0}^\infty p'\,|F(b_2)|\,(r_{2,k}-r_{1,k})\frac1{r_{2,k}}\nonumber\\
\mbox{ from \eqref{4} }&\ge& p'\,|F(b_2)|\, \frac{b_2-b_1}{C_{\lambda}}\sum_{k=n_0}^\infty \frac1{r_{2,k}}\;.\nonumber
\ee

But, setting $s_{2k-1}=r_{1,n_0+k-1},\,s_{2k}=r_{2,n_0+k-1},$ we have that $s_1<s_2<s_3<\cdots$ and for any $i$, $s_{i+1}-s_i\le 3\,T.$ Hence $s_n-s_1\le 3\,(n-1)\,T$, implying that
\[
s_n\le s_1+3\,(n-1)\,T
\]
and thus
\[
\frac1{s_n}\ge \frac1{s_1+3\,(n-1)\,T}\;.
\]
Therefore,
\[
\sum_{k=n_0}^\infty \frac1{r_{2,k}}=\sum_{k=1}^\infty \frac1{r_{2,n_0+k-1}}=\sum_{k=1}^\infty \frac1{r_{2,n_0+k-1}}= \sum_{k=1}^\infty \frac1{s_{2k}}\ge \sum_{k=1}^\infty \frac1{s_1+3\,(2\,k-1)\,T}=\infty
\]
contradicting the finiteness of the left hand side in \eqref{finite} and the proposition follows.
\end{proof}\vs

%---------------------------------------------------------------------
\begin{coro}\label{coro31} Under the assumptions of Proposition~\ref{basic2}, the only solutions $u_\lambda$ of~\eqref{ivp} satisfying $E_\lambda(r)\ge 0$ for all $r\ge 0$ are those that reach a double zero at some point $r_0>0$ and $u_\lambda(r)\equiv 0$ for all $r\ge r_0$. \end{coro}
%---------------------------------------------------------------------
\begin{proof} Let $u_\lambda$ be a solution of \eqref{ivp} such that $E_\lambda(r)\ge 0$ for all $r\ge0$, and assume that it does not reach any double zero. By Proposition~\ref{basic2}, $u_\lambda$ has at most a finite number of (simple) zeros. Without loss of generality we may assume that $u_\lambda(r)>0$ for $r>r_0$, for some $r_0>0$.

If $u_\lambda$ is eventually monotone, then $\lim_{r\to\infty}u_\lambda(r)=\ell$ exists, and thus by Proposition~\ref{basic1}(iv), $\ell$ is a zero of $f$ and $u_\lambda'\to 0$. By assumption $(\mathrm H3)$ \emph{i.e.} the compact support assumption, we know that $\ell\not=0$. Hence $\lim_{r\to\infty}E_\lambda(r)=F(\ell)<0$ because of $(\mathrm H4)$, implying that $E_\lambda(r)<0$ for $r$ sufficiently large.

If $u_\lambda$ has an infinite sequence of critical points, then in particular it has a first positive minimum at some point $r_1>0$. From the equation, $f(u_\lambda(r_1))\le 0$ and thus $0<u_\lambda(r_1)<A$, and thus $E_\lambda(r_1)=F(u_\lambda(r_1))<0$ implying that $E_\lambda(r)<0$ for all $r\ge r_1$.

Therefore, in both cases $u_\lambda$ must reach a first double zero at some $r_0>0$. As $E_\lambda$ decreases, it follows that $E_\lambda(r)=0$ for all $r\ge r_0$, and in particular, by differentiation,
\[
\Bigl((\phi_p(u_\lambda'))'+f(u_\lambda)\Bigr)\,u_\lambda'(r)=0\quad\mbox{for all $r\ge r_0$}\;,
\]
hence
\[
-\,\frac{N-1}{r}\,|u_\lambda'(r)|^p=0\quad\mbox{for all $r\ge r_0$}\;,
\]
implying that $u_\lambda'(r)=0$ for all $r\ge r_0$, thus $u_\lambda(r)=0$ for all $r\ge r_0$.
\end{proof}\vs

%---------------------------------------------------------------------
\begin{prop}\label{Pmm1} Let $f$ satisfy $(\mathrm H1)$-$(\mathrm H5)$ and let $u_\lambda$ be a solution of \eqref{ivp}. Let $\{s_n\}$ be any sequence in $[0,\infty)$ that tends to $\infty$ as $n\to\infty$ and define the sequence of real functions $\{v_n\}$ by
\[
v_n(r)=u_\lambda(r+s_n)\;.
\]
Then $\{v_n\}$ contains a subsequence that converges pointwise to a continuous function~$u_\lambda^\infty$, with uniform convergence on compact sets of $[0,\infty).$ Furthermore the function $u_\lambda^\infty$ is a solution to the asymptotic equation
\be\label{asym12}
(\phi_p(u'))'+f(u)=0\;.
\ee
Thus it satisfies
\[
(\phi_p({u_\lambda^\infty}'(r)))'+ f(u_\lambda^\infty(r))=0\;,
\]
for all $r\in [0,\infty)$. \end{prop}
%---------------------------------------------------------------------
\begin{proof} Let $u_\lambda$ be any solution to~(\ref{ivp}). We know that there exist two constants $c_\lambda^1$ and $c_\lambda^2$ such that
\[
u_\lambda(r)\leq c_\lambda^1\;,\quad u_\lambda'(r)\leq c_\lambda^2\;,\quad\text{for all }r\ge 0\;.
\]
Let now $\{s_n\}$ be any sequence in $[0,\infty)$ that tends to $\infty$ as $n\to\infty$ and define the sequence of real functions $\{v_n\}$ by
\[
v_n(r)=u_\lambda(r+s_n)\;.
\]
Then
\[
v_n(r)\leq c_\lambda^1\;,\quad v_n'(r)\leq c_\lambda^2\;,\quad\text{for all }r\ge 0\;.
\]
Hence, for any $s,t>0$, and all $n\in\N$,
\[
|v_n(s)-v_n(t)| \leq c_\lambda^2\,|s-t|\;.
\]
Then, from Ascoli's theorem (see \cite[Theorem 30]{MR1013117}), $\{v_n\}$ contains a subsequence, denoted the same, that converges pointwise to a continuous function $u_\lambda^\infty$, with uniform convergence on compact sets of $[0,\infty).$

It is clear that each function $v_n$ satisfies
\[
((r+s_n)^{N-1}\,\phi_p(v_n'(r)))'+(r+s_n)^{N-1} f(v_n(r))=0\;,
\]
and hence
\[
\phi_p(v_n'(r))=\phi_p(v_n'(0))-\int_0^r\Bigl(\dis\frac{t+s_n}{r+s_n}\Bigr)^{N-1}\,f(v_n(r))=0\;.
\]
By passing to a subsequence if necessary we can assume that $\phi_p(v_n'(0))\to a$ as $n\to \infty.$ Let now $T>0$, then since $\{f(v_n\}$ converges uniformly in $[0,T]$ to $f(u_\lambda^\infty)$, we find that $v_n'$ converges uniformly to a continuous function $z$ given by
\[
z(r)=\phi_{p'}\Bigl(a-\int_0^r f(u_\lambda^\infty(t))\;dt\Bigl)\;.
\]
Hence $z'$ exists and is continuous. Furthermore from
\[
v_n(r)=v_n(0) +\int_0^rv'_n(t)\;dt\;,
\]
letting $n\to\infty$, we obtain that
\[
u_\lambda^\infty(r)=u_\lambda^\infty(0)+\int_0^r z(t)\;dt\;.
\]
Hence $u_\lambda^\infty$ is continuously differentiable and ${u_\lambda^\infty}'(r)=z'(r),$ for all $r\in [0,T].$ Combining, we obtain
\[
\phi_p({u_\lambda^\infty}'(r))=a-\int_0^r f(u_\lambda^\infty(t))\;dt\;,
\]
that implies first that $a=\phi_p({u_\lambda^\infty}'(0)),$ and then that
\[
(\phi_p({u_\lambda^\infty}'(r)))'+ f(u_\lambda^\infty(r))=0\;.
\]
This argument show indeed that $u_\lambda^\infty$ is a solution to~(\ref{asym12}) for all $r\in [0,\infty)$.
\end{proof}

%---------------------------------------------------------------------
\begin{prop}\label{basic3} $\lim\limits_{r\to \infty} E_\lambda(r)=\mathcal E_\lambda=F(\ell)$, where $\ell$ is a zero of $f$. \end{prop}
%---------------------------------------------------------------------
\begin{proof} Let $T>0$ be arbitrary but fixed. Then
\ben
E_\lambda(k_0T)-\mathcal E_\lambda&=&(N-1)\int_{k_0T}^\infty\frac{|u'|^p}{t}\;dt\\
&=&(N-1)\sum_{k=k_0}^\infty\int_{k\,T}^{(k+1)\,T}\frac{|u'(t)|^p}{t}\;dt\\
&=&(N-1)\sum_{k=k_0}^\infty\int_{0}^{T}\frac{|u'(s+k\,T)|^p}{s+k\,T}\;ds\\
&\ge&(N-1)\sum_{k=k_0}^\infty\frac1{(k+1)\,T}\int_{0}^{T}|u'(s+k\,T)|^p\;ds\;.
\een
As the left hand side of this inequality is finite, it must be that
\[
\liminf_{k\to\infty}\int_{0}^{T}|u'(s+k\,T)|^p\;ds=0\;,
\]
hence there is a subsequence $\{n_k\}$ of natural numbers such that
\[
\lim_{k\to\infty}\int_{0}^{T}|u'(s+n_k\,T)|^p\;ds=0\;.
\]
{}From Proposition~\ref{Pmm1},
\[
v_k(r):=u(r+n_k\,T)
\]
has a subsequence, still denoted the same, such that
\[
\lim_{k\to\infty}v_k(r)=v(r)\quad\mbox{and}\quad \lim_{k\to\infty}v_k'(r)=v'(r)
\]
uniformly in compact intervals, where $v$ is a solution of
\[
(\phi_p(v'))'+f(v)=0\;.
\]

Hence,
\[
\int_0^T|v'(s)|^pds=0\;,
\]
implying that $v$ is a constant, say $v(r)\equiv v_0$. From the equation satisfied by $v$, $f(v_0)=0$. On the other hand,
\[
\frac{|v_k'(r)|^p}{p'}+F(v_k(r))=\frac{|u'(r+n_k\,T)|^p}{p'}+F(u(r+n_k\,T))=E_\lambda(r+n_k\,T)\to\mathcal E_\lambda
\]
as $k\to\infty$ and thus
\[
F(v_0)=\mathcal E_\lambda\;.
\]
\end{proof}

Although not necessary for the proof of our existence results in Theorem~\ref{Main} in our next result we give sufficient conditions for the limit of $u_\lambda(r)$ to exists as $r\to\infty$.\vs
%---------------------------------------------------------------------
\begin{thm}\label{limexis} Let $f$ satisfy $(\mathrm H1)$ through $(\mathrm H5)$, and assume furthermore that $f$ has only one positive zero at $a^+$ and only one negative zero at $b^-$. Then either
\[
\lim_{r\to\infty}u_\lambda(r)\quad\mbox{exists and equals either $a^+$ or $b^-$}\;,
\]
or $u_\lambda(r)\equiv 0$ for all $r\ge r_0$ for some $r_0>0$.

If $f$ has more than one positive or negative zero and if we assume that
\be\label{ext}
\int_{x_0}\frac{ds}{|F(s)-F(x_0)|^{1/p}}<\infty\quad\mbox{whenever $x_0$ is a local maximum of $F$}\;,
\ee
then $\lim_{r\to\infty}u_\lambda(r)$ exists and it is either a nonzero zero of $f$ or $u_\lambda(r)\equiv 0$ for all $r\ge r_0$, for some $r_0>0$. \end{thm}
%---------------------------------------------------------------------

\begin{proof} We first give the proof for the case $f$ has only one positive zero at $a^+$ and only one negative zero at $b^-$.

By Proposition~\ref{basic2} we can assume without loss of generality that $u_\lambda$ remains positive for $r> r_0$, for some $r_0>0$. If $u_\lambda$ has only a finite number of critical points, then it is eventually monotone and thus it converges as $r\to\infty$. Then the result follows from Proposition~\ref{basic1}$(iv)$.

Hence we are left with the case in which $u_\lambda$ has an infinite sequence of maxima at $\{r_n^M\}$ and an infinite sequence of minima at $\{r_n^m\}$, with both $u_\lambda(r_n^M), \ u_\lambda(r_n^m)>0$. From the equation, the maxima occur with $f(u_\lambda(r_n^M))\ge0$, hence $u_\lambda(r_n^M)>a^+$ (strict inequality due to Proposition~\ref{unique-ext} in Appendix~\ref{Section:Appendix}) and for the same reason, the minima occur with $f(u_\lambda(r_n^m))\le0$ with $u_\lambda(r_n^m)<a^+.$

As $E_\lambda$ is decreasing, we must have that $u_\lambda(r_n^m)$ increases (thus $u_\lambda(r_n^m)$ is bounded away from $0$) to a positive limit $\ell_1\in(0,a^+]$,  and $u_\lambda(r_n^M)$ decreases to a limit $\ell_2\in[a^+,A]$. Moreover,
\[
\liminf_{r\to\infty}u_\lambda(r)=\lim_{n\to\infty}u_\lambda(r_n^m)=\ell_1\;,\quad\mbox{and}\quad\limsup_{r\to\infty}u_\lambda(r)=\lim_{n\to\infty}u_\lambda(r_n^M)=\ell_2\;.
\]
Thus $E_\lambda(r_n^m)=F(u_\lambda(r_n^m))\to F(\ell_1)$ and $E_\lambda(r_n^M)=F(u_\lambda(r_n^M))\to F(\ell_2)$, implying $0\not=F(\ell_1)=F(\ell_2)$.

{}From Proposition~\ref{basic3}, $\lim_{r\to\infty}E_\lambda(r)$ is either $F(0)=0$ or $F(a^+)$. Since $0\not=F(\ell_1)$, the limit must be $F(a^+)$, and thus $F(\ell_1)=F(\ell_2)=F(a^+)$, and the only possibility is that $\ell_1=\ell_2=a^+$ proving the first part of the theorem.
\vs
In order to prove the second part of the theorem, for simplicity we consider $f$ with three positive zeros $u_1$, $u_2$ and $u_3$, but the arguments clearly hold for the general case. In this case $F$ has two minimum points at $u_1$ and $u_3$, and one maximum point at $u_2$, and the limit of the energy can be any of the three values $F(u_1)$, $F(u_3)$ or $F(u_2)$.

\medskip\noindent\emph{Claim 1:} If $\mathcal E_\lambda$ is a relative minima of $F$, then the solution $u_\lambda$ converges as $r\to\infty$.\smallskip

For the relative minima there are two cases: $F(u_1)=F(u_3)$ and $F(u_1)>F(u_3)$.

%---------------------------------------------------------------------
\begin{figure}[ht]
\begin{center}
\includegraphics[keepaspectratio, width=7cm]{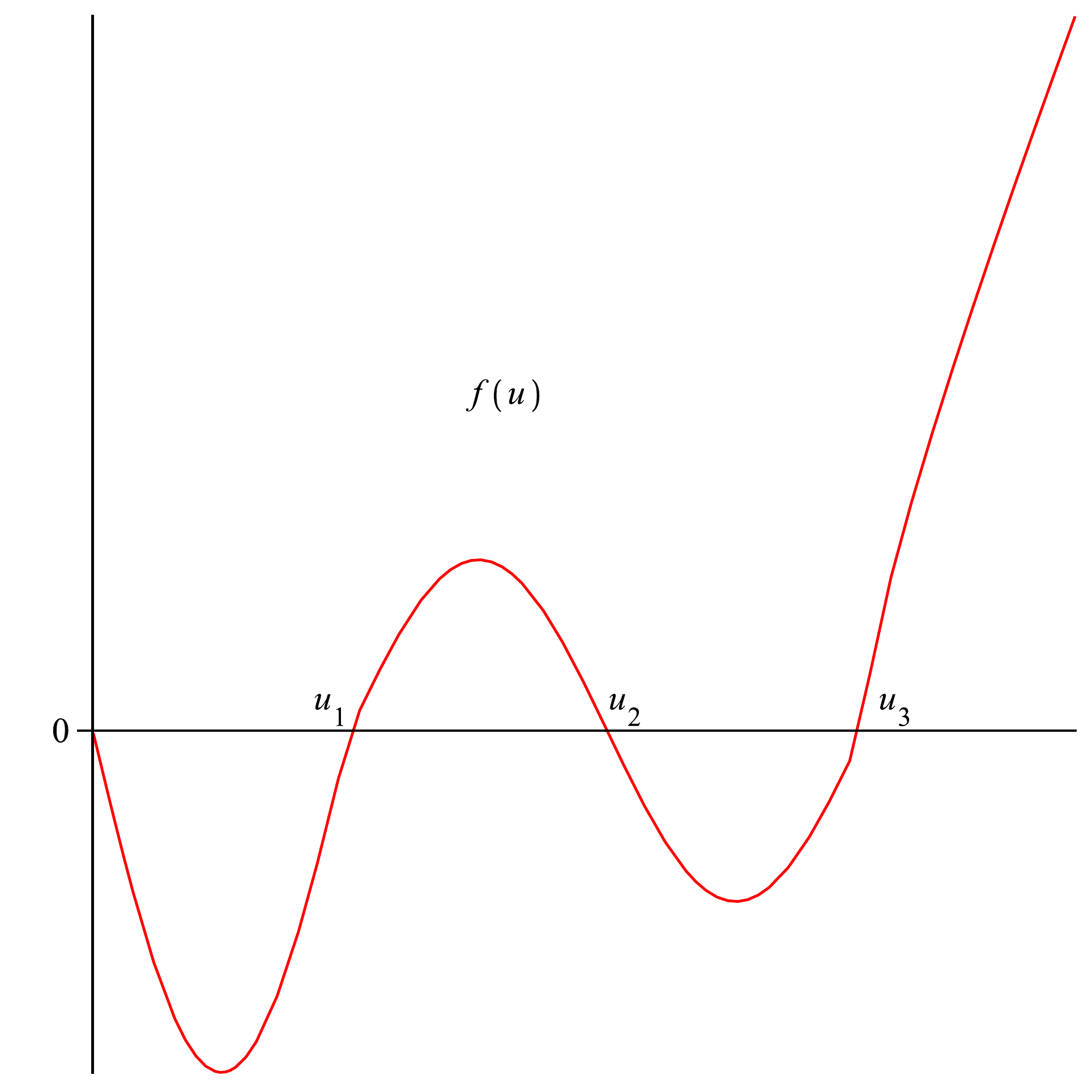}\quad \includegraphics[keepaspectratio, width=7cm]{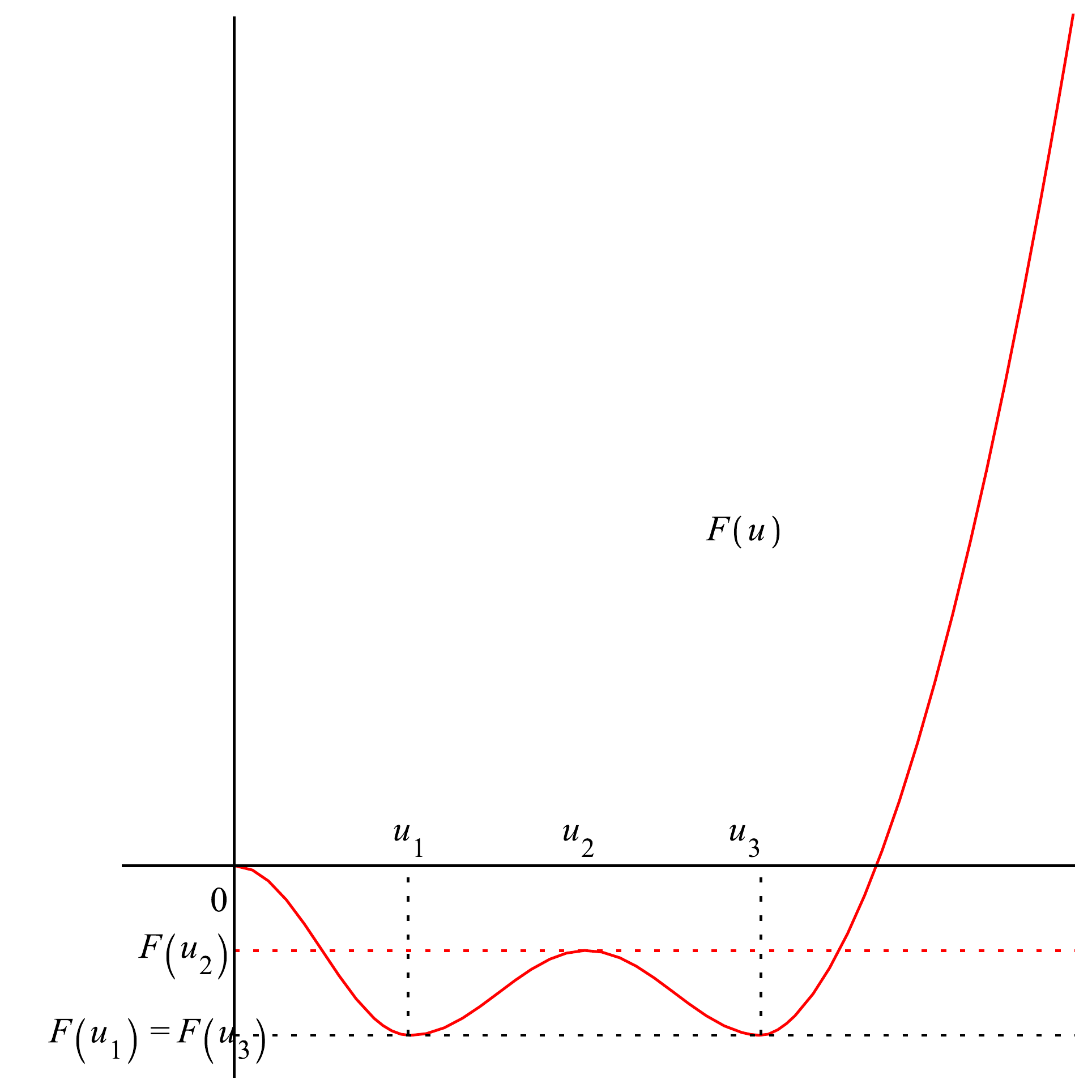}
\end{center}
\caption{Case of $f$ with three positive zeros and $F(u_1)=F(u_3).$}
\end{figure}
%---------------------------------------------------------------------

In the first case (shown in Figure 1), we can prove that if $E_\lambda$ converges to $L=F(u_1)=F(u_3)$, then the solution $u_\lambda$ either converges to $u_1$ or it converges to $u_3$. Indeed, we can assume that $u_\lambda(r)>0$ for $r\ge r_0$. If $u_\lambda$ has an infinite sequence of minima at $\{r_n^m\}$ and an infinite sequence of maxima at $\{r_n^M\}$ (tending to infinity), then by setting
\[
\ell_1=\liminf_{r\to\infty}u_\lambda(r)=\lim_{n\to\infty}u_\lambda(r_n^m)\;,\quad\ell_2=\limsup_{r\to\infty}u_\lambda(r)=\lim_{n\to\infty}u_\lambda(r_n^M)\;,
\]
we must have that
\[
F(\ell_1)=F(\ell_2)=L=F(u_1)\;,
\]
so if $\ell_1\not=\ell_2$, then $\ell_1=u_1$ and $\ell_2=u_3$. But then the solution $u_\lambda$ crosses the value $u_2$ at an infinite sequence $\{r_{2,n}\}$ tending to infinity and
\[
F(u_1)=\lim_{n\to\infty}E_\lambda(r_{2,n})=\lim_{n\to\infty}\frac{|u_\lambda'(r_{2,n})|^p}{p'}+F(u_2)\;,
\]
implying that
\[
\lim_{n\to\infty}\frac{|u_\lambda'(r_{2,n})|^p}{p'}=F(u_1)-F(u_2)<0\;,
\]
which is a contradiction. Hence $\ell_1=\ell_2$ and the claim follows.

The second case is a little more involved. The following two cases may occur:
\[
\mbox{(a)\quad}\lim_{r\to\infty}E_\lambda(r)=F(u_3)\quad\mbox{or}\quad\mbox{(b)\quad} \lim_{r\to\infty}E_\lambda(r)=F(u_1)\;.
\]

%---------------------------------------------------------------------
\begin{figure}[ht]
\begin{center}
\includegraphics[keepaspectratio, width=7cm]{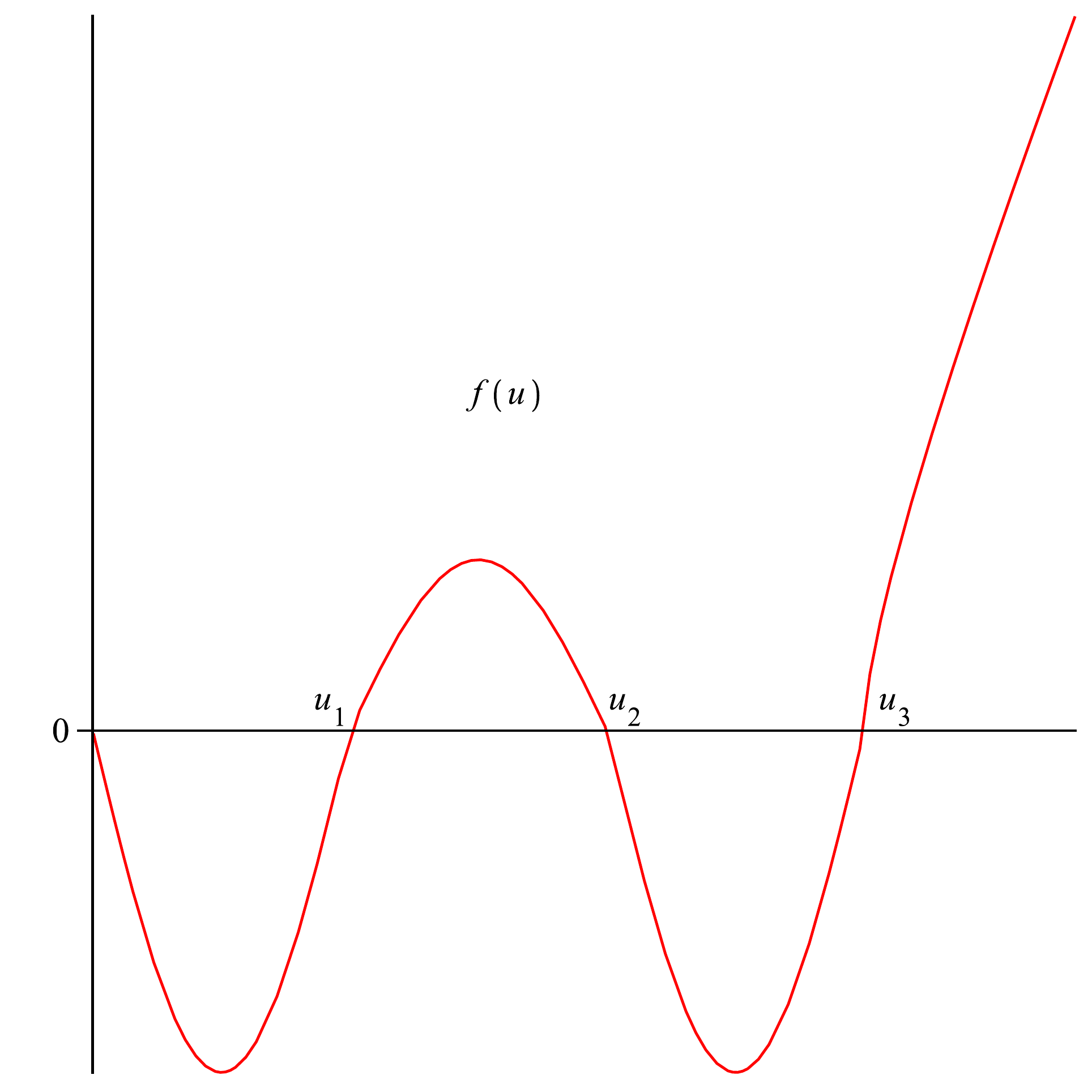}\quad \includegraphics[keepaspectratio, width=7cm]{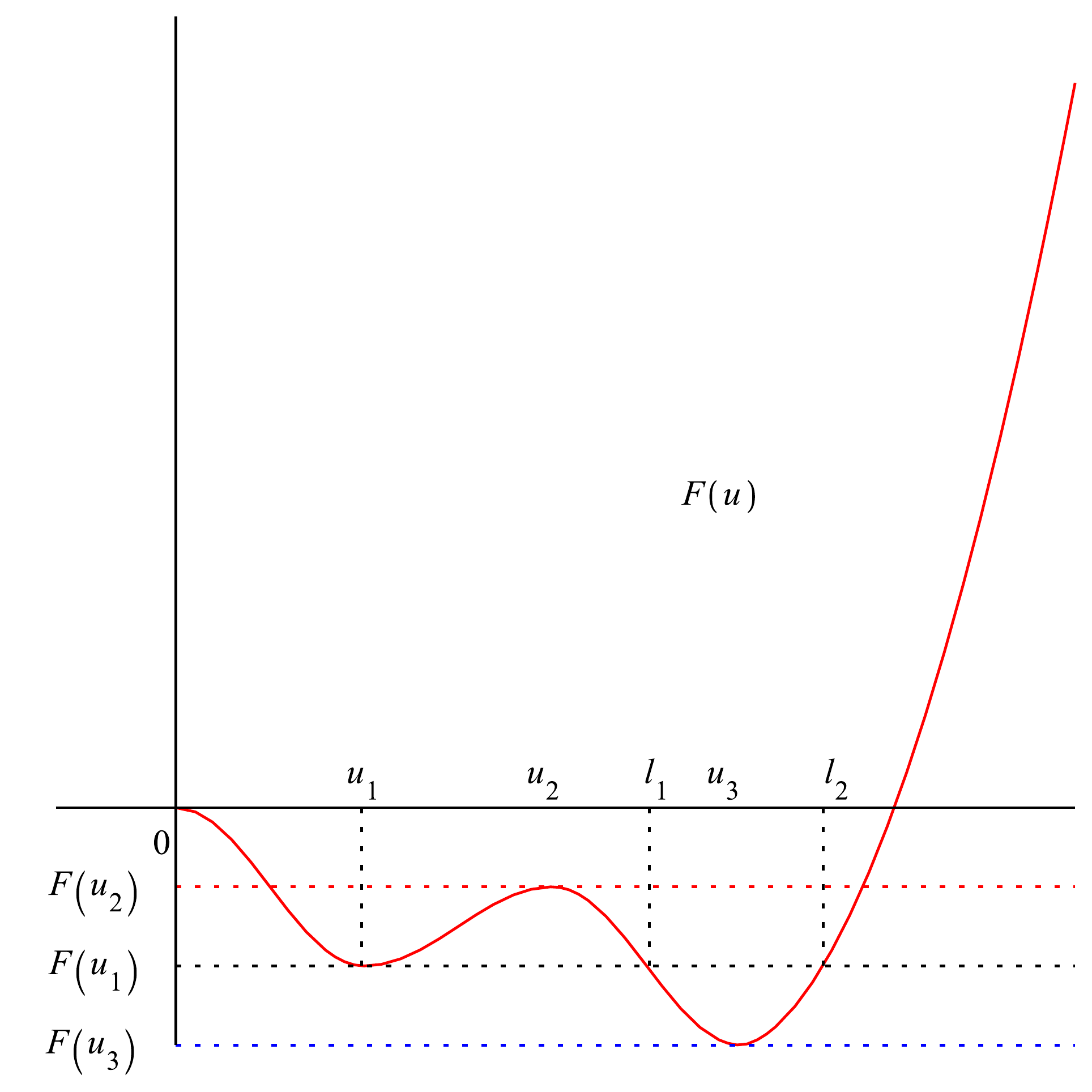}
\end{center}
\caption{Case of $f$ with three positive zeros and $F(u_1)>F(u_3).$}
\end{figure}
%---------------------------------------------------------------------

The case $(a)$ is simple because in this case $F(\ell_1)=F(\ell_2)=L=F(u_3)$ and the only possibility is that $\ell_1=\ell_2=u_3$.

In the second case we claim that $\ell_1=\ell_2=u_1$. If this is not true, then there are two possibilities: $(i)$ $\ell_1=u_1$ and $\ell_2$ as in Figure 2, or $(ii)$ $\ell_1$ and $\ell_2$ are as in the same figure. The first case is simple because again the solution $u_\lambda$ must cross the value $u_2$ at an infinite sequence $\{r_{2,n}\}$ tending to infinity and we arrive to the same contradiction as above.

For the second case, we proceed as in the proof of Proposition~\ref{basic2} and prove that the distance between any two consecutive critical points is bounded above. We set
\[
b_1=\frac{\ell_1+u_3}{2}\;,\quad b_2=\frac{\ell_2+u_3}{2}\;,
\]
and
let $r_{1,n}\in(r_n^m,r_n^M)$ be the unique point where $u_\lambda(r_{1,n})=b_1$, and $r_{2,n}\in(r_n^m,r_n^M)$ be the unique point where $u_\lambda(r_{2,n})=b_2$. See Figure 3.
%---------------------------------------------------------------------
\begin{figure}[ht]
\begin{center}
\includegraphics[keepaspectratio, width=10cm]{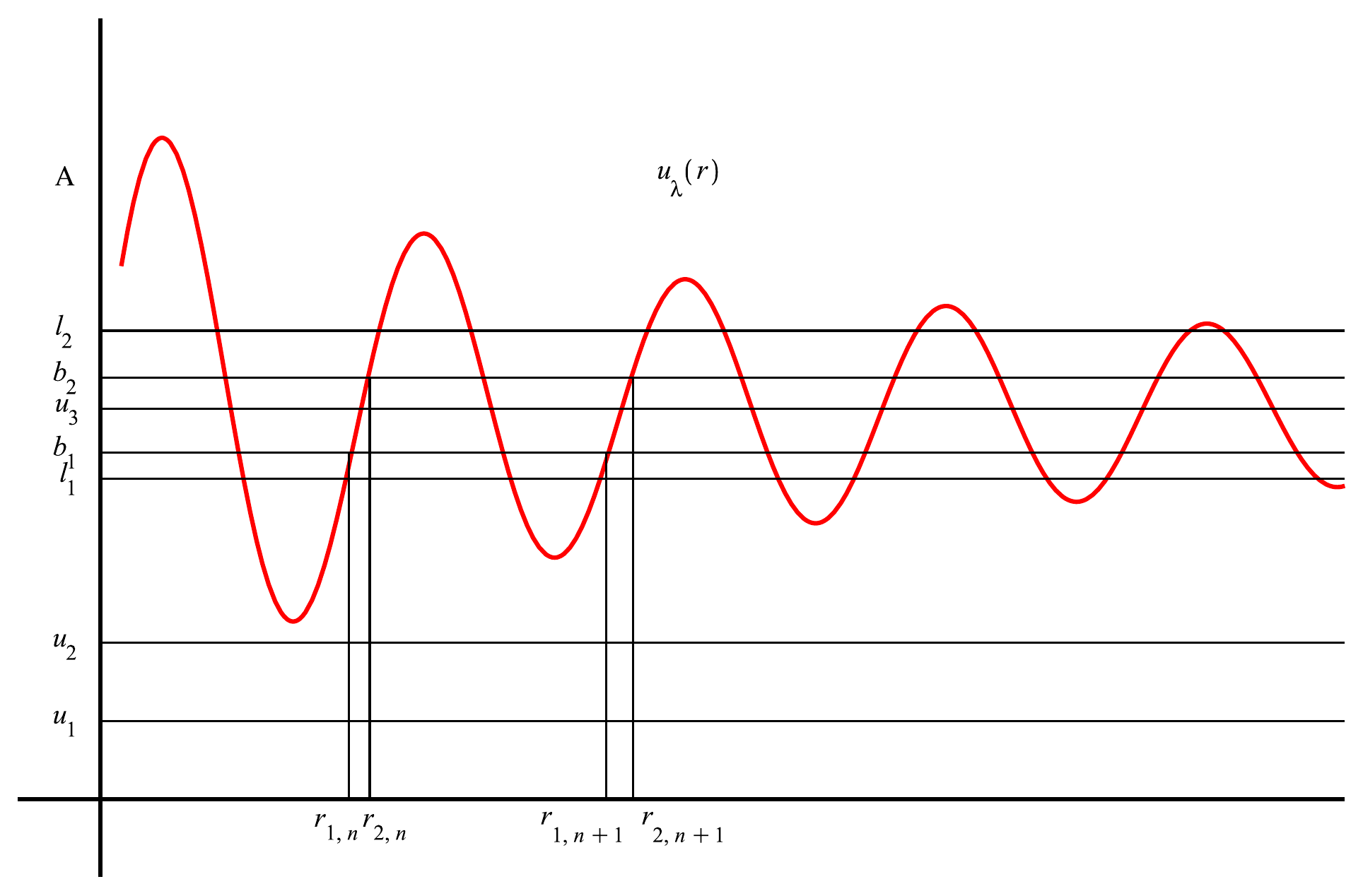}
\end{center}
\caption{Definition of the points $r_{1,n},$ $r_{2,n}.$}
\end{figure}
%---------------------------------------------------------------------

As the sequence $\{u_\lambda(r_n^m)\}$ increases to $\ell_1$, we may assume that $u_\lambda(r_n^m)\ge(u_2+\ell_1)/2$, and thus, $|f(u_\lambda(r))|$ is bounded below by some positive constant $c_1$ for all $r\in[r_n^m,r_{1,n}]$.
{}From the equation we have that for $r\in[r_n^m,r_{1,n}]$,
\ben
|(\phi_p(u_\lambda'))'(r)|&=&\Bigm|\frac{(N-1)}{r}\,\phi_p(u_\lambda'(r))+f(u_\lambda(r))\Bigm|\\
&\ge&|f(u_\lambda(r))|-\frac{(N-1)}{r}\,\phi_p(C_{\lambda})\\
&\ge&c_1-\frac{(N-1)}{r}\,\phi_p(C_{\lambda})\\
&\ge& \frac{c_1}{2}\quad\mbox{for all $r\ge\dis\frac{2\,(N-1)\,\phi_p(C_{\lambda})}{c_1}$}\;.
\een
Hence, choosing $n_0$ such that $r_n^m\ge \frac{2\,(N-1)\,\phi_p(C_{\lambda})}{c_1}$ for all $n\ge n_0$, we have that
\[
|(\phi_p(u_\lambda'))'(r)|\ge \frac{c_1}{2}\quad\mbox{for all }r\in[r_n^m,r_{1,n}]
\]
and therefore
\[
\phi_p(C_{\lambda})\ge \phi_p(u_\lambda'(r_{1,n}))-\phi_p(u_\lambda'(r_n^m))=(\phi_p(u_\lambda'))'(\xi)(r_{1,n}-r_n^m)\ge\frac{c_1}{2}\,(r_{1,n}-r_n^m)
\]
implying that
\be\label{m1}
r_{1,n}-r_n^m\le \frac{2\,\phi_p(C_{\lambda})}{c_1}\;.
\ee
Similarly, for $r\in[r_{2,n},r_n^M]$, using now that in this interval $f(u_\lambda(r))$ is bounded from below by a positive constant $c_2$, we conclude that there is $n_1\ge n_0$ such that
\be\label{m2}
r_n^M-r_{2,n}\le \frac{2\,\phi_p(C_{\lambda})}{c_2}
\ee
for all $n\ge n_1$.

Finally we estimate $r_{2,n}-r_{1,n}$. In the interval $[r_{1,n},r_{2,n}]$, $u_\lambda(r)\in[b_1,b_2]$ and $F(u_\lambda(r))\le max\{F(b_1),F(b_2)\}<F(u_1)$, hence there exists a positive constant $c_3$ such that
\[
F(u_1)-F(u_\lambda(r))\ge c_3\;,
\]
hence, using that $E_\lambda$ decreases to $F(u_1)$, we have that
\[
|u_\lambda'(r)|\ge (p'\,c_3)^{1/p}\;.
\]
Integrating this last inequality over $[r_{1,n},r_{2,n}]$, we obtain that
\be\label{m3}
r_{2,n}-r_{1,n}\le \frac{b_2-b_1}{(p'\,c_3)^{1/p}}\;.
\ee
Hence, from \eqref{m1}, \eqref{m2} and \eqref{m3}, we conclude that for all $n\ge n_1$
\[
r_n^M-r_n^m\le T
\]
where
\[
T=\frac{2\,\phi_p(C_{\lambda})}{c_2}+\frac{b_2-b_1}{(p'\,c_3)^{1/p}}+\frac{2\,\phi_p(C_{\lambda})}{c_1}\;.
\]

Again, from the mean value theorem
\[
b_2-b_1\le C_{\lambda}\,(r_{2,n}-r_{1,n})\;,
\]
hence, as before, we obtain the contradiction
\ben
F(\lambda)-F(u_1)>E_\lambda(r_{n_1}^m)-E_\lambda(\infty)&=&(N-1)\int_{r_{n_1}^m}^\infty\frac{|u_\lambda'(t)|^p}{t}\;dt\\
&\ge&(N-1)\sum_{k=n_1}^\infty\int_{r_{1,k}}^{r_{2,k}}\frac{|u_\lambda'(t)|^p}{t}\;dt \\
&\ge&(N-1)\sum_{k=n_1}^\infty p'\,c_3\,(r_{2,k}-r_{1,k})\frac1{r_{2,k}}\\
&\ge& p'\,c_3\, \frac{b_2-b_1}{C_{\lambda}}\sum_{k=n_1}^\infty \frac1{r_{2,k}}=\infty\;.
\een
Therefore, case $(ii)$ cannot happen and Claim 1 follows.

\medskip\noindent\emph{Claim 2:} If $\mathcal E_\lambda$ is a relative maxima of $F$, then the solution $u_\lambda$ converges as $r\to\infty$.\par\smallskip\noindent See Figure 4.
%---------------------------------------------------------------------
\begin{figure}[ht]
\begin{center}
\includegraphics[keepaspectratio, width=8cm]{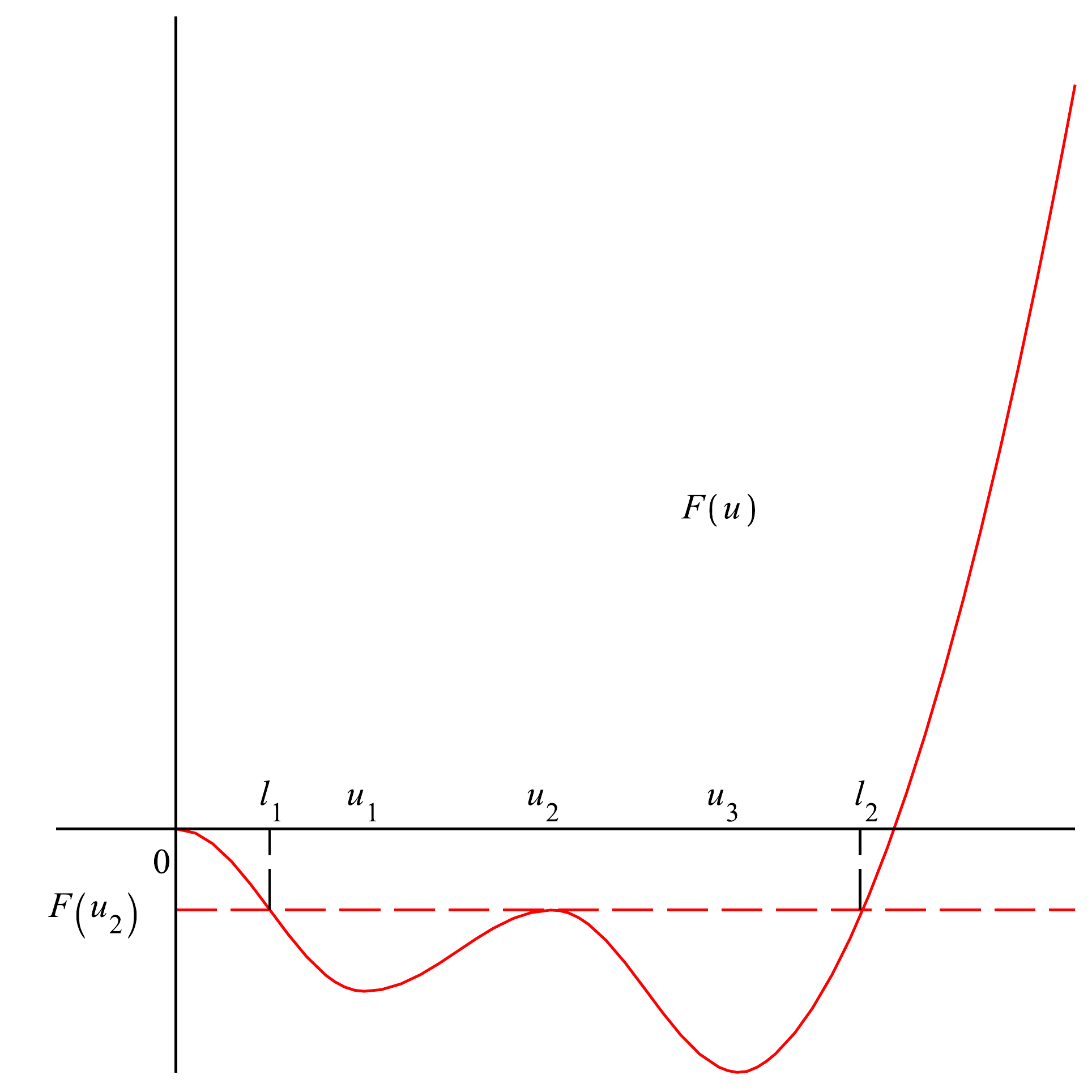}
\end{center}
\caption{If $\ell_1\not=\ell_2$, then $0<\ell_1<u_1$ and $u_3<\ell_2<A$.}
\end{figure}
%---------------------------------------------------------------------

{}From \eqref{ext},
\be\label{extra}
\int_{u_2}\frac{du}{|F(u_2)-F(u)|^{1/p}}\quad\mbox{is convergent}\;.
\ee
Then the same arguments used in the proof above can be used to establish the convergence of $u_\lambda$. Indeed, we let
\[
r_{1,n}: \quad u_\lambda(r_{1,n})=\frac{\ell_2+u_3}{2},\quad r_{2,n}:\quad u_\lambda(r_{2,n})=\frac{u_2+u_3}{2}\;,
\]
and for $r\in[r_{1,n},r_{2,n}]$, we have
$u_\lambda(r)\in[\frac{u_2+u_3}{2},\frac{\ell_2+u_3}{2}]$ and thus
\[
|u_\lambda'(r)|\ge (p')^{1/p}\,(F(u_2)-F(u_\lambda(r))\ge c_0>0
\]
for some positive constant $c_0$ implying that
\[
\frac{\ell_2-u_2}{2}\ge c_0\,(r_{2,n}-r_{1,n})
\]
and similarly, by setting
\[
\bar r_{1,n}: \quad u_\lambda(\bar r_{1,n})=\frac{u_1+u_2}{2}\;,\quad\bar r_{2,n}:\quad u_\lambda(\bar r_{2,n})=\frac{u_2+\ell_1}{2}\;,
\]
we have that
\[
\frac{u_2-\ell_1}{2}\ge c_0\,(\bar r_{2,n}-\,\bar r_{1,n})\;.
\]
For $r\in[r_{1,n},\bar r_{1,n}]$ we use \eqref{extra} to obtain that
\[
(p')^{1/p}\,(\bar r_{1,n}-r_{1,n})\le\int_{\frac{u_1+u_2}{2}}^{\frac{u_3+u_2}{2}}\frac{du}{|F(u_2)-F(u)|^{1/p}}\;.
\]
The bounds for $r_{2,n}-r_n^M$ and $r_{n+1}^m-\,\bar r_{2,n}$ is obtained as above using that $|f(u)|$ is bounded below in those intervals and using the equation to obtain
\[
\phi_p(C_\lambda)\ge c_0\,(r_{2,n}-r_n^M)\quad\mbox{and}\quad \phi_p(C_\lambda)\ge c_0\,(r_{n+1}^m-\,\bar r_{2,n})
\]
for some positive constant $c_0$.

We conclude that the distance between two consecutive critical points is bounded and we end the argument as we did at the end of Proposition~\ref{basic2}.
\end{proof}

%%%%%%%%%%%%%%%%%%%%%%%%%%%%%%%%%%%%%%%%%%%%%%%%%%%%%%%%%%%%%%%%%%%%%%
%%%%%%%%%%%%%%%%%%%%%%%%%%%%%%%%%%%%%%%%%%%%%%%%%%%%%%%%%%%%%%%%%%%%%%
\section{A change of coordinates and a lower bound on the angular velocity in the phase space}\label{Section:Coordinates}

In this section, we reformulate the problem in the phase space associated to the Hamiltonian system obtained in the ($p$)-linear case (that is, for $f(u)=|u|^{p-2}\,u$) in the asymptotic regime corresponding to $r\to\infty$. By computing a lower bound on the angular velocity around the origin, this will allow us to estimate the number of sign changes of the solutions, see
Section~\ref{Section:Existence}. First, let us explain how to change coordinates.

\medskip Setting $v=\phi_p(u')$, or equivalently $u'=\phi_{p'}(v)$, problem~(\ref{ivp}) is equivalent to the following first order system.
\beq\label{Flow}
\left\{\begin{array}{l} u'=\phi_q(v)\;,\cr\cr
v'=-\frac{N-1}r\,v-f(u)\;,\cr\cr
u(0)=\lambda\;,\quad v(0)=0\;.\end{array}\right.
\eeq
Here $q=p'$ stands for the H\"older conjugate of $p$. We consider also the auxiliary problem
\vs
\[\left\{
\begin{array}{l}
\dis\frac{dx}{dt}=-\phi_q(y)\;,\cr\cr
\dis\frac{dy}{dt}=\phi_p(x)\;,\cr\cr
x(0)=1\;,\quad y(0)=0\;. \end{array}
\right.\]
\vs
The auxiliary problem describes the asymptotics of \eqref{Flow} as $r\to\infty$, that is, when the $\frac{N-1}r\,v$ term is neglected in case of a ($p$)-linear function $f(u)=|u|^{p-2}\,u$. It is well known, see \cite{DEM}, that solutions to this last systems are $2\,\pi_p=2\,\pi_q$ periodic. Furthermore, with the notation of \cite{DEM}, we can define
\[
\sin_q(t):=y(t)\quad\text{and}\quad\cos_q(t):=x(t)=\phi_q\left(\dis\frac d{dt}\,\sin_q(t)\right)\;.
\]
It is immediate to check that
\[
\Phi_p(\cos_q(t))+\Phi_q(\sin_q(t))=\dis\frac1{p}\quad \text{for all}\quad t\in \R\;.
\]
To the $(u,v)$ coordinates of the phase plane, we assign \emph{generalized polar coordinates} $(\rho,\theta)$ by writing
\be\label{uvxy}
\left\{\begin{array}{l}
u=\rho^{\frac 1p}\cos_q(\theta)\cr\cr
v=\rho^{\frac 1q}\sin_q(\theta) \end{array}\right.
\ee
where
\[
\rho = p\,\left[\Phi_p(u)+\Phi_q(v)\right]\;.
\]
Notice that in case $p=q=2$, $(\sqrt\rho,\theta)$ are the usual polar coordinates of $(u,v)$, and $\cos_q$ and $\sin_q$ are the usual $\cos$ and $\sin$ functions.

\medskip Now, if $(u(r),v(r))$ denotes a solution to~(\ref{Flow}) and if we define the corresponding polar functions $r\mapsto\rho(r)$ and $r\mapsto\theta(r)$, then it turns out by direct computation that $(\rho,\theta)$ satisfies the following system of equations :
\begin{equation}\label{Syst:rho-theta}
\left\{\begin{array}{l}
\rho' = p\,\phi_q(v)\,\left[\phi_{p}(u)-f(u)-\frac{N-1}r\, v\right]\;,\cr\cr
\theta'=-\frac1\rho\left[ p\,\Phi_q(v)+u\,f(u)+\frac{N-1}r\,u\,v \right]\;,\cr\cr
\rho(0)=\lambda^p\;,\quad \theta(0)=0\;.\end{array}\right.
\end{equation}
We will denote by $(\rho_\lambda,\theta_\lambda)$ the solution of \eqref{Syst:rho-theta}.
\vs
The following lemma is a key step for our main result. We establish a lower bound on the angular velocity $|\theta'|$ around the origin, which will later allow us to estimate the number of sign changes of $u$ by counting the number of rotations of the solutions around the origin, in the phase plane. In order to formulate the lemma, we begin by noticing that from $(\mathrm H5)$, given $\omega\in(0,1/8)$ there is $s_0>0$ such that
\[
|f(s)|\ge4\,\omega\,|s|^{p-1}\quad\mbox{for all $|s|\ge s_0$\;}\;.
\]
%---------------------------------------------------------------------
\begin{lemma}[\rm Rotation Lemma]\label{AngularVelocity} With the previous notation, let assumptions $(\mathrm H1)$ through $(\mathrm H5)$ be satisfied and let $(\rho_\lambda, \theta_\lambda)$ be the generalized polar coordinates of a solution $(u_\lambda, v_\lambda)$ to~(\ref{Flow}). Set
\[
%r_0:=\frac{2\,(N-1)}{\omega\,(p-1)^{1/q}}\;,\quad\rho_0:=\max\left\{2\,s_0^p\;,\Big(4\,\sup_{x\in[-s_0,s_0]}|f(x)|\Big)^p\,\right\}
r_0:=\frac{2\,(N-1)}{\omega\,(p-1)^{1/q}}\;,\quad\sigma_0\ge\max\left\{2^{1/p}s_0\,,\;\Big(4\,\sup_{x\in[-s_0,s_0]}|f(x)|\Big)^{1/(p-1)}\,\right\}\;.
\]
Then, if $r\ge r_0$ and $\rho_\lambda\ge\sigma_0^p$, it holds that
\[
\theta_\lambda'(r)<-\,\omega\;.
\]
\end{lemma}
%---------------------------------------------------------------------
\begin{proof} We start by observing that with the above notation, \emph{i.e.}~$x=\cos_q(\theta)$,
\ben
-\,\theta'&=&\Bigl(1-|x|^p+\frac{x\,f(\sigma\,x)}{\sigma^{p-1}}\Bigr)+\frac{N-1}{r}\frac{u\,v}{\sigma^p}\nonumber\\
&\ge& \Bigl(1-|x|^p+\frac{x\,f(\sigma\,x)}{\sigma^{p-1}}\Bigr)-\frac{N-1}{r}\,|x|\,\Bigl(\frac{1-|x|^p}{p-1}\Bigr)^{1/q}\nonumber\\
&\ge& \Bigl(1-|x|^p+\frac{x\,f(\sigma\,x)}{\sigma^{p-1}}\Bigr)-\frac{N-1}{r}\frac1{(p-1)^{1/q}}\label{mv}
\een
where $\sigma=\rho^{1/p}$. It is clear that
\[
-\,\frac{N-1}{r}\frac1{(p-1)^{1/q}}>-\,\omega
\]
for any $r>r_0$. Hence in order to prove our result we need to estimate the minimum
\[
\min_{|x|\le 1}\Bigl[1-|x|^p+\frac{x\,f(\sigma\,x)}{\sigma^{p-1}}\Bigr]\;.
\]
First assume that $|x|^p\le 1/2$. If $\sigma\,|x|\ge s_0$, then
\ben
1-|x|^p+\frac{x\,f(\sigma\,x)}{\sigma^{p-1}}&=&1-|x|^p+\frac{|x|^p\,f(\sigma\,x)}{\sigma^{p-1}\,|x|^{p-2}\,x}\\
&\ge& 1-|x|^p+4\,\omega\,|x|^p\\
&&=1+(4\,\omega\,-1)\,|x|^p\ge\frac 12+2\,\omega\ge2\,\omega\;.
\een
Otherwise, if $\sigma\,|x|\le s_0$, then we have
\[
1-|x|^p+\frac{x\,f(\sigma\,x)}{\sigma^{p-1}}\ge 1-|x|^p-\frac{|x|}{\sigma^{p-1}}\sup_{s\in[-s_0,s_0]}|f(s)|>\frac1{2}-\frac1{\sigma^{p-1}}\sup_{s\in[-s_0,s_0]}|f(s)|>\frac14
\]
if $\sigma\ge \sigma_0$ and we already know that $\frac14>2\,\omega$.

On the other hand, if $|x|^p\ge 1/2$, then $\sigma\,|x|\ge 2^{-1/p}\,\sigma$, hence for $\sigma\ge\sigma_0$, we have $\sigma\,|x|\ge s_0$ and
\ben
1-|x|^p+\frac{x\,f(\sigma\,x)}{\sigma^{p-1}}&=&1-|x|^p+\frac{|x|^p\,f(\sigma\,x)}{\sigma^{p-1}\,|x|^{p-2}\,x}\\
&\ge& 1-|x|^p+4\,\omega\,|x|^p\ge2\,\omega\;.
\een
This concludes the proof.\end{proof}

In preparation for Section~\ref{Section:Existence} we finally relate the energy associated to the flow with the quantity $\rho$.
%---------------------------------------------------------------------
\begin{prop}\label{Energie-rho} Consider $E(u,v)=F(u)+\Phi_{p'}(v)$ and $\rho(u,v)=p\left[\Phi_p(u)+\Phi_{p'}(v)\right]$. Under assumptions $(\mathrm H1)$ through $(\mathrm H5)$, it holds that
\[
E(u,v)\to \infty\quad\text{if and only if }\quad \rho(u,v)\to \infty\;,
\]
for each $(u,v)$ in $\R^2$. \end{prop}
%---------------------------------------------------------------------
\begin{proof} The properties $E(u,v)\to\infty$, $\sup(|u|,|v|)\to\infty$, and $\rho(u,v)\to\infty$ are equivalent.\end{proof}

%%%%%%%%%%%%%%%%%%%%%%%%%%%%%%%%%%%%%%%%%%%%%%%%%%%%%%%%%%%%%%%%%%%%%%
%%%%%%%%%%%%%%%%%%%%%%%%%%%%%%%%%%%%%%%%%%%%%%%%%%%%%%%%%%%%%%%%%%%%%%
\section{Existence result}\label{Section:Existence}

We may now state our main result.
%---------------------------------------------------------------------
\begin{thm}\label{Main} Let $N>$, $p>1$ and suppose that assumptions $(\mathrm H1)$-$(\mathrm H6)$ are satisfied. Then there exists an unbounded increasing sequence $\{\lambda_k\}$ of initial data such that for any $k\in\N$,~(\ref{ivp}) with $\lambda=\lambda_k$, has a compactly supported solution with exactly $k$ nodes.\end{thm}
%---------------------------------------------------------------------
The proof is based upon some preliminary results that we state and prove next.\vs

For given $\lambda>A$, let $(u_\lambda, v_\lambda)$ be a solution to~(\ref{Flow}). Recall that the energy function~$E_\lambda$ has been defined by ~\eqref{funct-10}. For any $a\in [0,F(\lambda)]$, let us set
\[
r_\lambda(a):=\inf\{r\geq 0\,:\, E_\lambda(r)=a\}\;.
\]

We first observe that $r_\lambda(0)$ is finite. Indeed, if for some $\lambda\ge A$ (as defined in Section~\ref{Section:Intro}) we have that $r_\lambda(0)=\infty$, then $E_\lambda(r)\ge 0$ for all $r\ge 0$, and thus, from Corollary~\ref{coro31}, there exists $r_0>0$ such that $r_0$ is a double zero of $u_\lambda$ implying by the definition that $r_\lambda(0)\le r_0<\infty$.

We will denote by $N_{[0,R)}(\lambda)$ the number of nodes of $u_\lambda$ in $[0,R)$. For simplicity of notation, we will denote
\[
N(\lambda):=N_{[0,r_\lambda(0))}(\lambda)\;.
\]\vs

\noi Notice that all the possible zeros of $u_\lambda$ in $[0,r_\lambda(0))$ must be simple zeros.\vs

The following proposition was proved in \cite{GMZ}.
%---------------------------------------------------------------------
\begin{prop}\label{gamaza} Under assumptions $(\mathrm H1)$ through $(\mathrm H6)$, given $R>0$,
\[
\lim_{\lambda\to\infty}E_\lambda(r)=\infty
\]
uniformly for $r\in[0,R]$. \end{prop}
%---------------------------------------------------------------------
Now we start to make use of the variables introduced in Section~\ref{Section:Coordinates}.
%---------------------------------------------------------------------
\begin{prop}\label{Increasing} If $N(\lambda)>1$, then for any $r\in (0,r_\lambda(0))$, the number of nodes of~$u_\lambda$ in $(0,r)$ is given by
\[
\left[\Bigl(\frac{\pi_p}{2}-\,\theta_\lambda(r)\Bigr)\frac1{\pi_p}\right]
\]
where $[x]$ denotes the integer part of $x$. \end{prop}
%---------------------------------------------------------------------
\begin{proof} Follows directly from the change of variables~\eqref{uvxy}. \end{proof}\vs

Propositions~Ê\ref{gamaza} and~\ref{Increasing} combined with the \emph{Rotation Lemma}~\ref{AngularVelocity} on the angular velocity, yields the following result.

\begin{lemma}\label{lambdaBig} Under assumptions $(\mathrm H1)$ through $(\mathrm H6)$,
\[\lim_{\lambda\to+\infty}N(\lambda)=+\infty\;.\]
\end{lemma}\vs
\begin{proof} Let $M>0$. We will show that there exists $\lambda_M>0$ such that for $\lambda>\lambda_M$, we have $N(\lambda)>M$. We prove this by finding an interval $[0,R]$ with $R=R(M)$, such that the number of nodes in $[0,R]$ is greater than $M$. To do this we set
\[
R=\frac{\pi_p}{\omega}\left(M+\frac 12\right)+r_0\;.
\]
Using Propositions~\ref{Energie-rho} and~\ref{gamaza}, we know that there exists $\lambda_M$ such that for any $\lambda\ge\lambda_M$, $\rho_\lambda(r)>\rho_0:=\sigma_0^p$ in $[0,R]$. Next we apply the rotation Lemma~\ref{AngularVelocity} which ensures that
\[
-\,\theta_\lambda(R)\ge\omega\, R-\omega\, r_0-\,\theta_\lambda(r_0)\ge\omega\, R-\omega\, r_0=\left(M+\frac 12\right)\pi_p\;,
\]
by the choice of $R$. Applying Corollary~\ref{Increasing}, it follows that
\[
N(\lambda)\ge \left[M+1\right]>M\;.
\]
\end{proof}\vs

%---------------------------------------------------------------------
\begin{lemma}\label{NodeChanges} Under assumptions $(\mathrm H1)$ through $(\mathrm H5)$, when $\lambda>0$ varies, the number of nodes of the solution $u_\lambda$ can locally change by at most one. Moreover, to the value of~$\lambda$ at which the number of nodes changes corresponds a solution with compact support.\end{lemma}
%---------------------------------------------------------------------
\vs\noindent\emph{Proof.} This lemma is proved by defining for any $k\in\mathbb N_0:=\mathbb N\cup\{0\}$ the sets
\[\label{ak}
A_k:=\{\lambda\ge A\ :\ (u_\lambda(r),v_\lambda(r))\not=(0,0)\quad\mbox{for all }r\ge0\,,\mbox{ and } N(\lambda)=k\}\;,
\]
\[\label{ik}
I_k:=\{\lambda\ge A\ :\ (u_\lambda(r_\lambda(0)),v_\lambda(r_\lambda(0)))=(0,0)\quad\mbox{and } N_{[0,r_\lambda(0))}(\lambda)=k\}\;.
\]
Recall that $r_\lambda(a):=\inf\{r\geq 0\,:\,E_\lambda(r)=a\}$, and $E_\lambda$ has been defined by~\eqref{funct-10}. Notice that we have
$$
[A,\infty)=\big(\cup_{k\in\mathbb N_0}I_k\big)\cup\big(\cup_{k\in\mathbb N_0}A_k\big)\;.
$$
Indeed, let $\lambda\ge A$. Then $N(\lambda)=j$ for some $j\in\mathbb N_0$. If $u_\lambda(r_\lambda(0))\not=0$, then $u_\lambda$ does not have any double zero in $[0,\infty)$. Indeed, assume by contradiction that $r_1>r_\lambda(0)$ is a double zero of $u_\lambda$. Then by the monotonicity of $E_\lambda$, $E_\lambda(r)\equiv 0$ in $[r_\lambda(0),r_1]$. But then also $E_\lambda'(r)\equiv 0$ in $(r_\lambda(0),r_1)$ implying that $u_\lambda'(r)\equiv 0$ in $(r_\lambda(0),r_1)$ and thus $u_\lambda(r_\lambda(0))=u_\lambda(r_1)=0$, a contradiction. Hence $\lambda\in A_j$. If $u_\lambda(r_\lambda(0))=0$, then by the definition of $r_\lambda(0)$ we also have $u_\lambda'(r_\lambda(0))=0$ hence $\lambda\in I_j$. Also, observe that the sets $A_i$, $I_j$ are disjoint for any $i,\ j$, and for $i\not=j$, $A_i\cap A_j=\emptyset$ and $I_i\cap I_j=\emptyset$.

\medskip We also observe that if $\lambda\in A_j$, then necessarily $\lim_{r\to\infty}E_\lambda(r)<0$ (see Corollary~\ref{coro31}), and if $\lambda\in I_j$, then two cases may occur:
\[
\mbox{either }\lim_{r\to\infty}E_\lambda(r)<0\quad\mbox{or}\quad\lim_{r\to\infty}E_\lambda(r)=0\;.
\]
This due to the possible non-uniqueness of solutions to the initial value problem \eqref{ivp}, a solution could reach a double zero but not remain identically zero after that.

\medskip The proof of Lemma~\ref{NodeChanges} is a consequence of the following technical result.
%---------------------------------------------------------------------
\begin{prop} With the above notation, we have:
\begin{itemize}
\item[$(i)$] $A_k$ is open in $[A,\infty)$,
\item[$(ii)$] $A_k\cup I_k$ is bounded,
\item[$(iii)$] if $\lambda_0\in I_k$, then there exists $\delta>0$ such that $(\lambda_0-\delta,\lambda_0+\delta)\subset A_k\cup A_{k+1}\cup I_k$,
\item[$(iv)$] $\sup A_k\in I_{k-1}\cup I_k$, where we set $I_{-1}=\emptyset$ and,
\item[$(v)$] $\sup I_k\in I_k$.
\end{itemize}
\end{prop}
%---------------------------------------------------------------------
\begin{proof}
\noindent $(i)$ $A_k$ is open in $[A,\infty)$: Indeed, if $\bar\lambda\in A_k$, then in particular $(u_{\bar\lambda}(\bar r),v_{\bar\lambda}(\bar r))\not=(0,0)$, where $\bar r=r_{\bar\lambda}(0)$. Then there exists $\varepsilon_0>0$ such that the solution of \eqref{ivp} is unique in $[0,r_{\bar\lambda}(0)+\varepsilon]$ and $E_{\bar\lambda}(r_{\bar\lambda}(0)+\varepsilon/2)<0$ for all $\varepsilon\in(0,\varepsilon_0]$, and thus there exists $\delta>0$ such that
\[
E_{\lambda}(r_{\bar\lambda}(0)+\varepsilon/2)<0
\]
for all $\lambda\in(\bar\lambda-\delta,\bar\lambda+\delta)$ implying that $r_\lambda(0)\le r_{\bar\lambda}(0)+\varepsilon/2$. On the other hand, for the same reason, there exists $\delta'>0$ such that
\[
E_{\lambda}(r_{\bar\lambda}(0)-\varepsilon/2)>0
\]
for all $\lambda\in(\bar\lambda-\delta',\bar\lambda+\delta')$ implying that $r_\lambda(0)\ge r_{\bar\lambda}(0)-\varepsilon/2$. We conclude then that $r_\lambda(0)\to r_{\bar\lambda}(0)$. Hence the openness of $A_k$ follows from the continuous dependence of solutions in the initial value $\lambda$.

\medskip\noindent $(ii)$ The boundedness of $A_k\cup I_k$ is a consequence of Lemma~\ref{lambdaBig}.

\medskip\noindent $(iii)$ The proof of this statement follows that of \cite[Lemma 2.3]{Cortazar-GarciaHuidobro-Yarur}.
%---------------------------------------------------------------------
\begin{figure}[ht]
\begin{center}
\includegraphics[keepaspectratio, width=12cm]{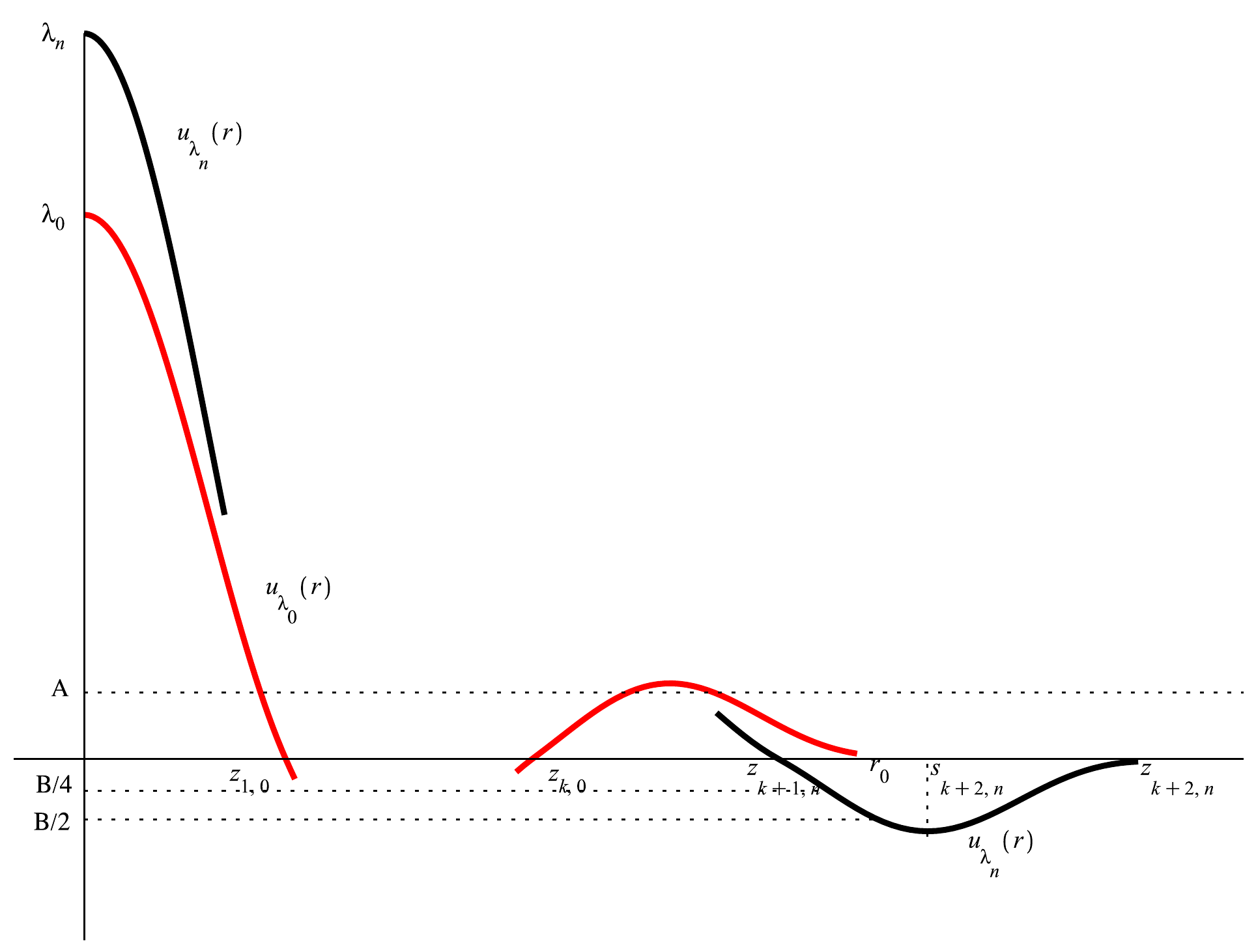}
\end{center}
\caption{Situation for $(iii)$.}
\end{figure}
%---------------------------------------------------------------------

Let $\lambda_0\in I_k$, set $r_0=r_{\lambda_0}(0)$ and let
\[
0<z_{1,0}<z_{2,0}<\ldots<z_{k,0}<r_0
\]
denote the $k$ zeros of $u_{\lambda_0}$ in $(0,r_0)$.

Assume first that $u_{\lambda_0}$ is decreasing in $(r_0-2\,\varepsilon_0,r_0)$ for some $\varepsilon_0>0$, so that it reaches a last maximum point at some $s_{k,0}\in(z_{k,0}, r_0)$. Let
\be\label{H}
H_\lambda(r):=r^{p'(N-1)}\,E_\lambda(r)\;.
\ee
As $u_{\lambda_0}(r_0)=0$ and $H_{\lambda_0}(r_0)=0$, given $\varepsilon>0$, there exists $\bar r<r_0$ such that
\[
0<u_{\lambda_0}(\bar r)<\frac A2\;,\quad\mbox{and}\quad H_{\lambda_0}(\bar r)<\varepsilon\;.
\]
Hence by continuous dependence of solutions to \eqref{ivp} in the initial data in any compact subset of $[0,r_0)$, there exists $\delta_0>0$ such that for $\lambda\in(\lambda_0-\delta_0,\lambda_0+\delta_0)$, the solution~$u_\lambda$ satisfies
\be\label{mm2}
0<u_\lambda(\bar r)<A\;,\quad H_\lambda(\bar r)<2\,\varepsilon\quad\mbox{and $u_\lambda$ has at least $k$ simple zeros in $[0,r_0)$}\;,
\ee
that is,
\be\label{mm1}
(\lambda_0-\delta_0,\lambda_0+\delta_0)\subset\big(\cup_{j\ge k}A_j\big)\cup\big(\cup_{j\ge k}I_j\big)\;.
\ee
Now we argue by contradiction and assume that there is a sequence $\{\lambda_n\}$ converging to $\lambda_0$ as $n\to\infty$ such that
\[
\lambda_n\not\in A_k\cup A_{k+1}\cup I_k\;.
\]
{}From \eqref{mm1},
\[
\lambda_n\in\big(\cup_{j\ge k+2}A_j\big)\cup\big(\cup_{j\ge k+1}I_j\big)\;,
\]
that is, the solution $u_{\lambda_n}$ has at least $k+2$ zeros and at least the first $k+1$ zeros are simple. Let us denote these zeros by
\[
0<z_{1,n}<z_{2,n}<\ldots<z_{k,n}<z_{k+1,n}<z_{k+2,n}\;.
\]
See Figure 5. By the choice of $\bar r$ and \eqref{mm2}, $u_{\lambda_n}$ decreases in $[\bar r, z_{k+1,n}]$. Let us denote by $s_{k+1,n}$ the point in $(z_{k+1,n},z_{k+2,n})$ where $u_{\lambda_n}$ reaches its minimum value. As $E_{\lambda_n}(z_{k+2,n})\ge0$, we must have that
\[
u_{\lambda_n}(s_{k+1,n})<B\;.
\]
Let us denote by $r_{1,n}<r_{2,n}$ the unique points in $(z_{k+1,n},s_{k+1,n})$ where
\[
u_{\lambda_n}(r_{1,n})=\frac{B}4\;,\quad u_{\lambda_n}(r_{2,n})=\frac{B}{2}\;.
\]
{}From \eqref{H}, we have that
\be\label{Hder}
H'_{\lambda_n}(r)=p'\,(N-1)\,r^{p'(N-1)-1}\,F(u_{\lambda_n}(r))\;.
\ee
Therefore, using the first estimate in \eqref{mm2}, we have that for $n$ large enough, $H'_{\lambda_n}( r)<0$ for $r\in[\bar r,z_{k+1,n}]$ and thus by the second in \eqref{mm2}, $H_{\lambda_n}(z_{k+1,n})<2\,\varepsilon$\;.

Integrating now \eqref{Hder} over $[z_{k+1,n},r_{2,n}]$, and using that $F(u_{\lambda_n}(t))<0$ in this range and $p'(N-1)-1= \frac{p}{p-1}(N-1)-1\ge p-1>0$, we find that
\ben
H_{\lambda_n}(r_{2,n})-H_{\lambda_n}(z_{k+1,n})&=&-\,p'\,(N-1)\int_{z_{k+1,n}}^{r_{2,n}}t^{p'(N-1)-1}|F(u_{\lambda_n}(t))|\;dt\\
&\le&-\,p'\,(N-1)(z_{k+1,n})^{p'(N-1)-1}\int_{z_{k+1,n}}^{r_{2,n}}|F(u_{\lambda_n}(t))|\;dt\\
&\le&-\,p'\,(N-1)(z_{k+1,n})^{p'(N-1)-1}\int_{r_{1,n}}^{r_{2,n}}|F(u_{\lambda_n}(t))|\;dt\\
&\le&-\,C\,p'\,(N-1)(z_{k+1,n})^{p'(N-1)-1}(r_{2,n}-r_{1,n})\;,
\een
where
\[
C:=\inf_{s\in[\frac{B}{2},\frac{B}4]}|F(s)|\;.
\]
But from the mean value theorem, and for $n$ large enough, we have that
\[
\frac{|B|}4=|u_{\lambda_n}(r_{2,n})-u_{\lambda_n}(r_{1,n})|\le C_{\lambda_0+1}\,(r_{2,n}-r_{1,n})\;,
\]
and
\[
\frac{\lambda_0}{2}\le \lambda_n=|u_{\lambda_n}(0)-u_{\lambda_n}(z_{1,n})|\le C_{\lambda_0+1}\,z_{1,n}\le C_{\lambda_0+1}\,z_{k+1,n}\;,
\]
hence
\ben
H_{\lambda_n}(r_{2,n})&\le& 2\,\varepsilon-C\,p'\,(N-1)\,(z_{k+1,n})^{p'(N-1)-1}\,\frac{|B|}{4\,C_{\lambda_0+1}}\\
&\le&2\,\varepsilon-C\,p'\,(N-1)\,\Bigl(\frac{\lambda_0}{2\,C_{\lambda_0+1}}\Bigr)^{p'(N-1)-1}\,\frac{|B|}{4\,C_{\lambda_0+1}}\;.
\een
By choosing from the beginning $\varepsilon\in\Bigl(0,C\,p'\,(N-1)\Bigl(\frac{\lambda_0}{2\,C_{\lambda_0+1}}\Bigr)^{p'(N-1)-1}\frac{|B|}{8C_{\lambda_0+1}}\Bigr)$ we obtain that
\[
H_{\lambda_n}(r_{2,n})= r_{2,n}^{p'(N-1)}E_{\lambda_n}(r_{2,n})<0\;,
\]
contradicting the fact that
\[
E_{\lambda_n}(r_{2,n})\ge E_{\lambda_n}(z_{k+2,n})\ge 0\;.
\]

A similar computation provides a contradiction if we assume that $u_{\lambda_0}$ is increasing in $(r_0-2\,\varepsilon_0,r_0)$ for some $\varepsilon_0>0$. Altogether $(iii)$ is established.

\medskip\noindent $(iv)$ Assume next that $A_k\not=\emptyset$, let $\lambda_0=\sup A_k$ and set $r_0=r_{\lambda_0}(0)$. As $A_j$ is open for every $j\in\mathbb N_0$, $\lambda_0\not\in A_j$ for any $j$ hence $\lambda_0\in I_j$ for some $j$, and by continuous dependence of the solutions in the initial data in $[0,r_0-\varepsilon]$ for $\varepsilon>0$ small enough, $j\le k$. By $(iii)$, there is $\delta>0$ such that $(\lambda_0-\delta,\lambda_0]\subset A_j\cup A_{j+1}\cup I_j$, and since $A_k\cap(\lambda_0-\delta,\lambda_0]\not=\emptyset$, it must be that
\[
A_k\cap (A_j\cup A_{j+1}\cup I_j)\not=\emptyset\;,
\]
hence $j=k$ or $j=k-1$.

\medskip\noindent $(v)$ $\sup I_k\in I_k$: It follows directly from $(iii)$.
\end{proof}

\bigskip\begin{proof}[Proof of Theorem~\ref{Main}] With the notation of the previous lemma one shows by induction that there exists an increasing sequence $\{\lambda_k\}$, $\lambda_k\to +\infty$, such that $\lambda_k\in I_k$.

As $A\in A_0$, by $(ii)$ we can set $\lambda_0=\sup A_0$, and by $(iv)$ and $(v)$, $\lambda_0\in I_0$ and $\lambda_0\le \sup I_0\in I_0$. We use now $(iii)$ and find $\delta>0$ such that
\[
(\sup I_0-\delta, \sup I_0+\delta)\subset A_0\cup A_1\cup I_0\;.
\]
Since $(\sup I_0,\sup I_0+\delta)\cap A_0=\emptyset$ by the definition of $\lambda_0$ and $(\sup I_0,\sup I_0+\delta)\cap I_0=\emptyset$ by the definition of $\sup I_0$, it must be that
\[
(\sup I_0,\sup I_0+\delta)\subset A_1
\]
implying that
\[
A_1\not=\emptyset\quad\mbox{and}\quad \lambda_0\le \sup I_0<\lambda_1:=\sup A_1\;.
\]
By $(iv)$, $\lambda_1\in I_0\cup I_1$, but as $\sup I_0<\lambda_1$, it must be that $\sup A_1\in I_1$. Then
\[
I_1\quad\mbox{is not empty and}\quad \lambda_1\le \sup I_1\;.
\]
We use again $(iii)$ to find $\delta>0$ such that
\[
(\sup I_1-\delta,\sup I_1+\delta)\subset A_1\cup A_2\cup I_1\;,
\]
and again deduce that
\[
(\sup I_1,\sup I_1+\delta)\subset A_2\;,
\]
hence $A_2\not=\emptyset$ and thanks to $(ii)$ we can set $\lambda_2=\sup A_2$, $\lambda_0<\lambda_1\le \sup I_1<\lambda_2$ and $\lambda_2\in I_2$. We continue this procedure to obtain the infinite strictly increasing sequence $\{\lambda_k\}$, defined by $\lambda_k=\sup A_k$ with $\lambda_k\in I_k$.
\end{proof}

%%%%%%%%%%%%%%%%%%%%%%%%%%%%%%%%%%%%%%%%%%%%%%%%%%%%%%%%%%%%%%%%%%%%%%
%%%%%%%%%%%%%%%%%%%%%%%%%%%%%%%%%%%%%%%%%%%%%%%%%%%%%%%%%%%%%%%%%%%%%%
\section{Qualitative properties of the solutions}\label{Section:Qualitative}

Several qualitative properties can be deduced from our intermediate results and from their proofs. Without entering the details let us summarize the most striking ones.

When $\lambda$ varies, the number of nodes changes of at most one. The energy of any solution decreases as $r$ increases and converges to a finite limit as $r\to\infty$. More precisely, solutions are of two types: either the limit of their energy is negative  or the limit of the energy is zero, and the corresponding solutions are compactly supported.

Solutions which have a double zero can be compactly supported or not, as in this case uniqueness may be lost.

\medskip For solutions with compact support, the size of the support increases with the number of nodes, and diverges as the number of nodes goes to infinity. This is a consequence of  Lemma \ref{AngularVelocity} and Proposition~\ref{gamaza}, as can be easily proved arguing by contradiction. With the generality of Theorem~\ref{Main}, it is not easy to give quantitative results but one can estimate the size of the support of the solutions and the number of nodes for large values of $\lambda$, as the following proposition shows.
%---------------------------------------------------------------------
\begin{prop}\label{Size} Let $1<p<N$. Under the assumptions of Theorem~\ref{Main}, we have that $r_\lambda(0)$ and $N(\lambda)$
%the size of the support and the number of nodes of any solution of~\eqref{ivp},
are bounded below by
 $$C \Bigl(\frac{\lambda^{\frac{N(p-1)}{N-p}}}{f(\lambda)}\Bigr)^{\frac{N-p}{p(N-1)}} $$
  where $C$ is a positive constant independent of $\lambda$.

  In particular, if $\lim\limits_{\lambda\to\infty}\dis\frac{{\lambda^{\frac{N(p-1)}{N-p}}}}{f(\lambda)}=\infty$, then $r_\lambda(0)\to\infty$ and $N(\lambda)\to\infty$ as $\lambda\to\infty.$
 \end{prop}
%---------------------------------------------------------------------
\begin{proof} Let $\theta\in (0,1)$ be as in (H6), and for $\lambda>0$, let $S_{\theta,\lambda}:=\inf\{r>0\,:\,u_\lambda(r)=\theta\,\lambda\}$. It can be easily shown that a solution of~\eqref{ivp} satisfies
\[
(1-\theta)\,\lambda = \lambda-u_\lambda(S_{\theta,\lambda}) =\int_0^{S_{\theta,\lambda}}\left(\frac 1{r^{N-1}}\int_0^rs^{N-1}f(u_\lambda(s)) \; ds\right)^{1/(p-1)}\;dr\;.
\]
As a consequence of the monotonicity of $u_\lambda$ in $[0, S_{\theta,\lambda}]$ and (H5), for $\theta\lambda$ large enough we obtain
\be\label{ult}
N^{1/(p-1)}\,p'\,\frac{(1-\theta)\,\lambda} {\left[f(\lambda)\right]^{1/(p-1)}} \leq \left(S_{\theta,\lambda}\right)^{p'} \leq N^{1/(p-1)}\,p'\,\frac{(1-\theta)\,\lambda} {\left[f(\theta\,\lambda)\right]^{1/(p-1)}}\;.
\ee
On the other hand, since
\[
\frac d{dr}(E_\lambda+{\bar F})=-\frac{N-1}r\,|v|^{p'} \geq-\frac{N-1}r\,p'\,(E_\lambda+{\bar F})\;,
\]
for $S_{\theta,\lambda}\leq r\leq r_\lambda(0)$, we obtain
\[
\frac{E_\lambda(r)+{\bar F}}{F(\theta\,\lambda)+{\bar F}}\geq \left(\frac{S_{\theta,\lambda}}r\right)^{(N-1)\,p'}.\label{dcyEny}
\]
Hence
\ben
r_\lambda(0)^{p'}&\geq& S_{\theta,\lambda}^{p'} \left(\frac{F(\theta\,\lambda)}{{\bar F}}\right)^{\frac{1}{N-1}}\\
&\geq&\frac{1}{\bar F^{1/(N-1)}}S_{\theta,\lambda}^{p'}\left(\frac{F(\theta\,\lambda)}{{\lambda f(\lambda)}}\right)^{\frac{1}{N-1}}(\lambda f(\lambda))^{\frac{1}{N-1}}\\
\mbox{from (H6)}&\geq& CS_{\theta,\lambda}^{p'}(\lambda f(\lambda))^{\frac{1}{N-1}}\\
\mbox{from the first in \eqref{ult}}&\geq& C\Bigl(\frac{\lambda^{\frac{N(p-1)}{N-p}}}{f(\lambda)}\Bigr)^{\frac{N-p}{(p-1)(N-1)}}.
\een
Hence,
$$r_\lambda(0)\ge C \Bigl(\frac{\lambda^{\frac{N(p-1)}{N-p}}}{f(\lambda)}\Bigr)^{\frac{N-p}{p(N-1)}}.$$
This proves the assertion on the size of the support. The conclusion on the number of nodes follows from Lemma~\ref{AngularVelocity}.\end{proof}

%%%%%%%%%%%%%%%%%%%%%%%%%%%%%%%%%%%%%%%%%%%%%%%%%%%%%%%%%%%%%%%%%%%%%%
%%%%%%%%%%%%%%%%%%%%%%%%%%%%%%%%%%%%%%%%%%%%%%%%%%%%%%%%%%%%%%%%%%%%%%
\appendix\section{}\label{Section:Appendix}

In this Appendix, for sake of completeness, we state some basic results concerning the initial value problem \eqref{ivp}. We begin with a result on the existence of solutions.
%---------------------------------------------------------------------
\begin{prop}\label{Existence} Suppose that assumption $(\mathrm H1)$ holds. If $\lim_{u\to\pm\infty}f(u)=\pm\infty$, for any fixed $\lambda\in\R$, then~(\ref{ivp}) has a solution defined in $[0,\infty)$. \end{prop}
%---------------------------------------------------------------------
\begin{proof} As before let $F(u)=\int_0^uf(s)ds$, then $F(u)\to \infty$ as $u\to\pm\infty$. Recall that
\[
-\,\bar F=\inf\limits_{u\in\R}{F(u)}\;,
\]
and suppose that $u$ is a solution to~(\ref{eq2}) such that $u(0)=\lambda.$ Then $u$ satisfies
\[
\frac{|u'(r)|^p}{p'}\le \frac{|u'(r)|^p}{p'}+F(u(r))+\bar F\le F(\lambda)+\bar F\;.
\]
Hence $|u'|\le (p'\,C_\lambda)^{1/p}$ with $C_\lambda=F(\lambda)+\bar F$ and $|u|\le |\lambda|+(p'\,C_\lambda)^{1/p}\,r$ for $r>0$ in the domain of definition of $u$. These estimates tell us that if $u$ can be defined in an interval of the form $[0,\delta]$ for $\delta>0$ and small, then this solution can be extended to $[0,\infty).$
\vs
Consider therefore the Banach space $\mathcal C:=C([0,\delta];\R)$ of continuous functions on $[0,\delta]$, endowed with the sup norm $\| \cdot \|_{\infty}$. A solution in $\mathcal C$ will exist if and only if the operator~$T$ defined on $\mathcal C$ by
\be\label{OperT}
T(u)(r):=\lambda-\int_0^r\phi_{p'}\left(\int_0^\tau\frac{s^{N-1}}{\tau^{N-1}}\,f(u(s))\;ds\right)\;d\tau
\ee
has a fixed point. For $\varepsilon>0$ given, let $B(\lambda,\varepsilon)$ be the ball in $\mathcal C$ with center $\lambda$ and radius~$\varepsilon$.
Then if $u\in \overline {B(\lambda,\varepsilon)}$ we have that for all $r\in [0,\delta]$ it holds that $-\varepsilon+\lambda \leq u(r)\leq\varepsilon+\lambda.$ Let us set $m:=\max\limits_{|u-\lambda|\leq\varepsilon}|f(u)|$. Then from~(\ref{OperT}), we find the estimate
\[
|T(u)(r)-\lambda |\leq \displaystyle\frac{\delta^{p'}\,m^{p'-1}}{p'\,N^{p'-1}}\;.
\]
If $\delta$ is small so that $\frac{\delta^{p'} m^{p'-1}}{p'N^{p'-1}}\leq \varepsilon$, we have that $T(\overline {B(\lambda,\varepsilon)})\subset \overline {B(\lambda,\varepsilon)}$.

To show that $T$ is completely continuous, let $\{u_k\}$ be a sequence in $\overline {B(\lambda,\varepsilon)}$ and consider $s,$ $t\in[0,\delta]$. Then
\[
| T(u_k)(t)-T(u_k)(s) |\leq \displaystyle\frac{\delta^{p'-1}\,m^{p'-1}}{N^{p'-1}}\,|t-s|\;.
\]
{}From Ascoli-Arzela theorem it follow that $T$ is compact on $\overline {B(\lambda,\varepsilon)}$. To show that $T$ is continuous let $\{u_k\}$ be a sequence in $\overline {B(\lambda,\varepsilon)}$ such that $u_k\to u\in \overline {B(\lambda,\varepsilon)}$, as $k\to\infty.$ An application of Lebesgue's dominated theorem to
\[
T(u_k)(r):=\lambda-\int_0^r\phi_{p'}\left(\int_0^\tau\frac{s^{N-1}}{\tau^{N-1}}\,f(u_k(s))\;ds\right)\;d\tau\;,
\]
shows that $T(u_k)\to T(u)$ in $\mathcal C$ as $k\to\infty.$

Then by the Schauder fixed point theorem, $T$ possesses a fixed point in $\overline {B(\lambda,\varepsilon)}$ which is what we wanted to prove.
\end{proof}\vs

The last proposition states a unique extendibility result for solutions to the initial value problem \eqref{ivp}; this result has been proved in \cite{Cortazar-GarciaHuidobro-Yarur}.
%---------------------------------------------------------------------
\begin{prop}\label{unique-ext} Let $f$ satisfy $(\mathrm H1)$-$(\mathrm H2)$. Then solutions to \eqref{ivp} are unique at least until they reach a double zero or a point $u_0=u_\lambda(r_0)$, where $u_\lambda'(r_0)=0$ and $u_0$ is a local maximum of $F$. \end{prop}
%---------------------------------------------------------------------

%%%%%%%%%%%%%%%%%%%%%%%%%%%%%%%%%%%%%%%%%%%%%%%%%%%%%%%%%%%%%%%%%%%%%%
%%%%%%%%%%%%%%%%%%%%%%%%%%%%%%%%%%%%%%%%%%%%%%%%%%%%%%%%%%%%%%%%%%%%%%
% \bibliographystyle{amsalpha}
% \bibliography{References}

\newcommand{\etalchar}[1]{$^{#1}$}
\providecommand{\bysame}{\leavevmode\hbox to3em{\hrulefill}\thinspace}
\providecommand{\MR}{\relax\ifhmode\unskip\space\fi MR }
\providecommand{\MRhref}[2]{%
  \href{http://www.ams.org/mathscinet-getitem?mr=#1}{#2}
}
\providecommand{\href}[2]{#2}

\end{document}